\documentclass[12pt]{article}
\usepackage{srcltx}
\usepackage{eurosym}
\usepackage{mathtools}
\usepackage{amsmath}
\usepackage{amsfonts}
\usepackage{amssymb}
\usepackage{amsthm}
\usepackage{graphicx}
\usepackage{mathrsfs}
\usepackage{xcolor}
\usepackage{exscale}
\usepackage{latexsym}
\usepackage[colorlinks,plainpages=true,pdfpagelabels,hypertexnames=true,colorlinks=true,pdfstartview=FitV,linkcolor=blue,citecolor=red,urlcolor=black]{hyperref}
\PassOptionsToPackage{unicode}{hyperref}
\PassOptionsToPackage{naturalnames}{hyperref}
\usepackage{enumerate}
\usepackage[shortlabels]{enumitem}
\usepackage{bookmark}
\usepackage{wasysym}
\usepackage{esint}
\usepackage[ddmmyyyy]{datetime}
\usepackage[margin=1in]{geometry}
\makeatletter
\g@addto@macro\normalsize{%
	\setlength\abovedisplayskip{4pt}
	\setlength\belowdisplayskip{4pt}
	\setlength\abovedisplayshortskip{4pt}
	\setlength\belowdisplayshortskip{4pt}
}
\everymath{\displaystyle}
\usepackage[capitalize,nameinlink]{cleveref}
\crefname{section}{Section}{Sections}
\crefname{subsection}{Subsection}{Subsections}
\crefname{condition}{Condition}{Conditions}
\crefname{hypothesis}{Hypothesis}{Conditions}
\crefname{assumption}{Assumption}{Assumptions}
\crefname{lemma}{Lemma}{Lemmas}
\crefname{definition}{Definition}{Definitions}

\crefformat{equation}{\textup{#2(#1)#3}}
\crefrangeformat{equation}{\textup{#3(#1)#4--#5(#2)#6}}
\crefmultiformat{equation}{\textup{#2(#1)#3}}{ and \textup{#2(#1)#3}}
{, \textup{#2(#1)#3}}{, and \textup{#2(#1)#3}}
\crefrangemultiformat{equation}{\textup{#3(#1)#4--#5(#2)#6}}%
{ and \textup{#3(#1)#4--#5(#2)#6}}{, \textup{#3(#1)#4--#5(#2)#6}}%
{, and \textup{#3(#1)#4--#5(#2)#6}}

\Crefformat{equation}{#2Equation~\textup{(#1)}#3}
\Crefrangeformat{equation}{Equations~\textup{#3(#1)#4--#5(#2)#6}}
\Crefmultiformat{equation}{Equations~\textup{#2(#1)#3}}{ and \textup{#2(#1)#3}}
{, \textup{#2(#1)#3}}{, and \textup{#2(#1)#3}}
\Crefrangemultiformat{equation}{Equations~\textup{#3(#1)#4--#5(#2)#6}}%
{ and \textup{#3(#1)#4--#5(#2)#6}}{, \textup{#3(#1)#4--#5(#2)#6}}%
{, and \textup{#3(#1)#4--#5(#2)#6}}

\crefdefaultlabelformat{#2\textup{#1}#3}

\newtheorem{theorem} {Theorem}[section]

\newtheorem{remarka}{Remark}
\newtheorem{lemma}[theorem]{Lemma}

\newtheorem{counter example}[theorem]{Counter Example}
\newtheorem{remark}[theorem] {Remark}
\newtheorem{definition}[theorem] {Definition}

\newtheorem{claim}[theorem] {Claim}

\def\N{\mathbb{N}}
\def\CC{{\rm \kern.24em \vrule width.02em height1.4ex depth-.05ex \kern-.26emC}}

\def\TagOnRight

\def\AA{{it I} \hskip-3pt{\tt A}}

\def\QQ{\rlap {\raise 0.4ex \hbox{$\scriptscriptstyle |$}} {\hskip -0.1em Q}}

\makeatletter
\newcommand{\vo}{\vec{o}\@ifnextchar{^}{\,}{}}
\makeatother

\def\YYint#1#2#3{{\setbox0=\hbox{$#1{#2#3}{\iint}$}
		\vcenter{\hbox{$#2#3$}}\kern-.50\wd0}}

\def\XXint#1#2#3{{\setbox0=\hbox{$#1{#2#3}{\int}$}
		\vcenter{\hbox{$#2#3$}}\kern-.50\wd0}}

\makeatletter
\def\namedlabel#1#2{\begingroup
	\def\@currentlabel{#2}%
	\label{#1}\endgroup
}
\makeatother
\makeatletter
\newcommand{\rmh}[1]{\mathpalette{\raisem@th{#1}}}
\newcommand{\raisem@th}[3]{\hspace*{-1pt}\raisebox{#1}{$#2#3$}}
\makeatother

\newcounter{desccount}

\newcommand{\descref}[2]{\hyperref[#1]{\textnormal{\textcolor{black}{}\textcolor{blue}{\bf #2}\textcolor{black}{}}}}

\newcommand{\dref}[2]{\hyperref[#1]{\textcolor{black}{(}\textcolor{blue}{\bf #2}\textcolor{black}{)}}}
\newcommand{\be} {\begin{eqnarray}}
\newcommand{\ee} {\end{eqnarray}}
\newcommand{\Bea} {\begin{eqnarray*}}
	\newcommand{\Eea} {\end{eqnarray*}}
\newcommand{\pa} {\partial}

\newcommand{\al} {\alpha}
\newcommand{\rr}{\rightarrow}

\newcommand{\B} {\beta}
\newcommand{\de} {\delta}

\newcommand{\p}  {\prime}
\newcommand{\e}  {\epsilon}
\newcommand{\De} {\Delta}

\newcommand{\la} {\lambda}
\newcommand{\si} {\sigma}

\newcommand{\f}{\infty}
\newcommand{\R}{\mathbb{R}}

\newcommand{\noi} {\noindent}

\newcommand{\ga}{\gamma}

\newcommand{\norm}[1]{\left|\hspace{-0.2mm}\left| #1 \right|\hspace{-0.2mm}\right|}
\newcommand{\abs}[1]{\left| #1\right|}

\newcounter{whitney}
\refstepcounter{whitney}

\newcounter{ineqcounter}
\refstepcounter{ineqcounter}
\makeatletter
\def\ps@pprintTitle{%
	\let\@oddhead\@empty
	\let\@evenhead\@empty
	\def\@oddfoot{}%
	\let\@evenfoot\@oddfoot}
\makeatother
\usepackage[doublespacing]{setspace}
\usepackage[titletoc,toc,page]{appendix}

\makeatletter
\newcommand{\refcheckize}[1]{%
	\expandafter\let\csname @@\string#1\endcsname#1%
	\expandafter\DeclareRobustCommand\csname relax\string#1\endcsname[1]{%
		\csname @@\string#1\endcsname{##1}\wrtusdrf{##1}}%
	\expandafter\let\expandafter#1\csname relax\string#1\endcsname
}
\makeatother

\refcheckize{\cref}
\refcheckize{\Cref}


\makeatletter
\newcommand{\mainsectionstyle}{%
	\renewcommand{\@secnumfont}{\bfseries}
	\renewcommand\section{\@startsection{section}{2}%
		\z@{.5\linespacing\@plus.7\linespacing}{-.5em}%
		{\normalfont\bfseries}}%
}
\makeatother
\usepackage{pgf,tikz}
\usetikzlibrary{arrows}
\usetikzlibrary{decorations.pathreplacing}

\usepackage{xpatch}
\xpatchcmd{\MaketitleBox}{\hrule}{}{}{}
\xpatchcmd{\MaketitleBox}{\hrule}{}{}{}

\date{}
 
\makeatletter

\title{Existence of BV solution for the Euler-Poisson system in one dimension with  large initial data
}
\author{Shyam Sundar Ghoshal\thanks{Centre for Applicable Mathematics, Tata Institute of Fundamental Research, Post Bag No 6503, Sharadanagar, Bangalore - 560065, India. E-mail: ghoshal@tifrbng.res.in}, Boris Haspot\thanks{Universit\'e Paris Dauphine, PSL Research University, Ceremade, Umr Cnrs 7534, Place du Mar\' echal De Lattre De Tassigny 75775 Paris cedex 16 (France), E-mail: haspot@ceremade.dauphine.fr} \,and Animesh Jana\thanks{Universit\'e Paris Dauphine, PSL Research University, Ceremade, Umr Cnrs 7534, Place du Mar\' echal De Lattre De Tassigny 75775 Paris cedex 16 (France), E-mail: animesh.jana@dauphine.psl.eu}}

\begin{document}
	
	\maketitle
	\begin{abstract}			
				This paper deals with the existence of BV solution for the Euler-Poisson system endowed with a $\gamma$ pressure law. More precisely, we prove the existence of weak solution in the BV framework with arbitrary large initial data when $\gamma=1+2\e$ satisfies a smallness condition.  We use the Glimm scheme combined with a splitting method as introduced in [Poupaud, Rascle and Vila, J. Differential Equations, 1995]. Existence of BV solution of  1-D isentropic Euler equation for large data and $\ga=1+2\e$ is proved in [Nishida and Smoller, Comm. Pure Appl. Math, 1973]. Due to the presence of electric field, the difficulty arises while controlling the Glimm functional for the Euler-Poisson system. It requires a subtle study of wave interaction. In the later part of this article, we  discuss the initial-boundary value problem for the Euler-Poisson system. We prove the existence of $BV$ solution for the initial-boundary value problem with large initial and boundary data. By an explicit example, we also show ill-posedness of initial-boundary value problem for the isentropic Euler equation.
%
	\end{abstract}
\tableofcontents	


\section{Introduction}\label{sec:intro}
 %
 We consider the following Euler-Poisson system, 
\begin{align}
\frac{\pa}{\pa t}\varrho+\frac{\pa}{\pa x}(\varrho u)&=0,\label{eqn-EP-1}\\
\frac{\pa}{\pa t}(\varrho u)+\frac{\pa}{\pa x}(\varrho u^2+P(\varrho))&=-\si \varrho u-\frac{q}{m}\varrho \Psi,\label{eqn-EP-2}\\
\frac{\pa}{\pa x}\Psi&=-\frac{q}{e}(\varrho-\mu),\label{eqn-EP-3}
\end{align}	
for $t>0$ and $x\in\R$. It appears in various models describing physical phenomena such as transport of electron, plasma collision (see \cite{1,MRS,Poupaud-1,Poupaud-2}). In the system \eqref{eqn-EP-1}--\eqref{eqn-EP-3} $\varrho,u$ represent respectively the concentration and mean velocity of electrons whereas the physical constants $q,m,e $ stand for the electric charge, mass of electron and the permittivity of the medium respectively. In equation \eqref{eqn-EP-2}, $P(\varrho)$ is the pressure-density relation and we consider $\gamma$--law in the sequel with $P(\varrho)=\varrho^\gamma$ where $\gamma=1+2\e$ and $\e>0$.
We supplement the previous system with the following Cauchy data:
\begin{equation}\label{eq:initial-data}
\begin{cases}
\varrho(0,x)&=\varrho_0(x),\,x\in\R,\\
u(0,x)&=u_0(x),\,x\in\R.
\end{cases}
\end{equation}
In this article, we want to show the existence of BV solution to Euler-Poisson system \eqref{eqn-EP-1}--\eqref{eqn-EP-3} with large initial data \eqref{eq:initial-data}. For a given time $T$ and BV data $(\varrho_0,u_0)$ we prove that there exists $\ga_0$ such that for any $\ga\in(1,\ga_0)$ the system \eqref{eqn-EP-1}--\eqref{eqn-EP-3} admits a  BV solution on $[0,T]$ corresponding to the initial data $(\varrho_0,u_0)$. Moreover, we study the existence of BV solution to the initial boundary value problem for \eqref{eqn-EP-1}--\eqref{eqn-EP-3} with large BV data. 

 	 The study of mathematical aspects of the Euler-Poisson system has been started in the works of Tatarski\u{\i} \cite{Tatarskii} and Degond-Markowich \cite{DM}. In the article \cite{DM}, authors proved the existence of solution to the steady state Euler-Poisson system in 1-D for the subsonic case, that is, $\abs{u}<c$, where $c:=(p^\p(\varrho))^{1/2}$ denotes the sound speed. The steady state problem of Euler-Poisson system has been studied \cite{AMPS,DiMichele-Marcati-et-al,Gamba,Rosini} by using various methods, e.g. viscosity method \cite{Gamba}, phase plane analysis \cite{AMPS,Rosini}. Local existence of strong solution to Euler-Poisson system in multi-D is known \cite{DLYY} for initial data in suitable Sobolev space. Global results can be found in dimension 2 and 3 in \cite{Guo1,Ionescu,Li} for small perturbation of the state $(\bar{\rho},0)$ with $\bar{\rho}>0$ and the initial velocity is assumed to be irrotational. More recently, Guo, Han and Zhang \cite{Guo} have 
 	 proved that global strong solution exists if the initial data is sufficiently small in Sobolev and weighted Sobolev spaces with sufficiently large index of regularity in one dimension for $\gamma=3$. Concerning the existence of global solution for initial density close from the vacuum, we refer in particular to \cite{Dan1,Had} and the reference therein. In particular in \cite{Had} the authors prove the global existence of free-boundary Euler-Poisson solutions satisfying \textit{the physical vacuum condition}. For a hyperbolic system, it is well known that smooth data in general may not give global smooth solution. In a system like \eqref{eqn-EP-1}--\eqref{eqn-EP-3} with hyperbolic-elliptic mixed structure, the same property holds \cite{E96,Perthame}. Therefore, we look for existence of weak solution. 
 	 
 	 In this article, we are interested in investigating the existence problem in BV framework. In his seminal work \cite{Glimm}, Glimm proved the existence of weak solution to hyperbolic systems of conservation laws with genuinely nonlinear or linearly degenerate characteristic fields for initial data having small total variation bound. 
 	  Later Lax and Glimm in \cite{GL-1} extended this previous result  to $2\times2$ hyperbolic systems  with genuinely nonlinear fields providing that the initial data is only bounded with a smallness condition on the $L^\infty$ norm. Bressan \cite{Bressan-1} developed a new strategy for proving Glimm's result for $n\times n$ hyperbolic systems by using the so called wave front tracking method. This tool is essential to prove the existence of a semi-group of contraction $L^1$ of solutions, it enables in particular to obtain the uniqueness of the Glimm's solution in a class of uniqueness which takes into account the Lax entropy conditions (we refer to \cite{Br} and the references therein for these questions).
	  For hyperbolic system of conservation laws in 1-D vanishing viscosity limit is obtained by Bianchini and Bressan \cite{BB} for small BV data. Nishida \cite{Nishida} proved the existence of global BV weak solution of p-system with $\ga=1$ for large BV data. For $\ga>1$, Nishida and Smoller \cite{Nisi-Smo} showed the existence of global weak solution to 1-D isentropic Euler system in Lagrangian coordinate (known as $p$-system) for arbitrary large BV data by considering a smallness condition on $(\ga-1)/{2}$. The system under consideration \eqref{eqn-EP-1}--\eqref{eqn-EP-3} shares some similarities with hyperbolic balance laws. The existence problem for hyperbolic balance laws has been studied for small BV data via Glimm scheme \cite{DaHs}, wave front tracking method  \cite{AGG}, fixed-point argument \cite{CG-2} and operator splitting method \cite{AG,CG-1}.
 	 
 	 For Euler-Poisson system BV solutions are less studied. Poupaud, Rascle and Vila \cite{P-R-V} have obtained the existence of BV solution for large initial data in the case of the isothermal Euler-Poisson system which corresponds to the pressure law $P(\varrho)=\varrho$ that is $\ga=1$. They use a splitting method combined with Glimm scheme  and the use of the  Nishida functional introduced in \cite{Nishida}  for getting the BV stability of the approximated solution. Besides, there are results \cite{MN,MN-1,ZhangB} concerning the global existence of bounded weak solutions for the Euler-Poisson system \eqref{eqn-EP-1}--\eqref{eqn-EP-3}. These existence results are based on the method of compensated compactness adapted to Euler-Poisson system. Though these results  \cite{MN,MN-1,ZhangB} is valid for $L^\f$ initial data, they fail to provide any information on regularity of solutions. We note that for hyperbolic balance laws, the uniqueness is known \cite{AGG,AG,CG-1,CG-2} only for small BV data. It is not known whether the solutions arising via compensated compactness \cite{MN,MN-1,ZhangB} are uniquely determined by initial condition. We are interested in establishing the existence of BV solution for \eqref{eqn-EP-1}--\eqref{eqn-EP-3}. We also study the existence BV solution to initial-boundary value problem for the Euler-Poisson system when the initial and the boundary data have large total variations. The existence and uniqueness results on initial-boundary value problem for general hyperbolic systems are also limited to small BV setting \cite{AC,Amadori,DM-1}. We would like to point out that the BV setting is well-suited for boundary value problems as it guarantees the solution to have boundary trace and satisfy the boundary condition (see \cite{Amadori,NS-bdy,Nishida}). Since stability and uniqueness results for hyperbolic systems are known only for BV solutions, we keep our focus on proving the existence of BV solution for Euler-Poisson system when initial data has finite (possibly large) total variation.

In this article, we prove existence of BV solutions for initial data with large total variation. Proof is based on two steps: (i) approximating the solution via splitting method with an adaptation of Glimm scheme, (ii) estimating Glimm functional to get TV bound of solution. Concerning the approximate part, we modify the scheme developed in \cite{P-R-V} with necessary changes. This has been required in order to get proper interaction estimates. Note that TV bound for $P(\varrho)=\varrho$ case is achieved by considering TV of Riemann invariant, more precisely, it has been observed in \cite{Nishida} that TV bound of Riemann invariant decays in time. For $\ga>1$ case, the TV norm of Riemann invariant can increase for isentropic Euler system but with an order $\e \abs{\al}\abs{\B}$ where $\e=(\ga-1)/2$ and $\al,\B$ are the strengths of two shocks interacting. Therefore, we need to consider Glimm functional to get bound as in Nishida and Smoller \cite{Nisi-Smo}. For Euler-Poisson equation, at each time step there is an increment due to electric field which has been bounded with $C_1\De t$ where $C_1$ depends on the TV norm of initial data \cite{P-R-V}. Since for $\ga>1$ we are working with quadratic functional, it needs a non-trivial treatment to obtain TV bound. Note that at each step of the numerical scheme,  taking into account the electric field the $\varrho$ variable is unchanged compared with the Glimm scheme for the isentropic Euler system whereas the velocity unknown  $u$ is translated for some $\de$ and due to this modification, the total variation can increase in a more significant way as in \cite{Nisi-Smo}. With a careful analysis of Riemann problem and interaction waves, we are able to show that change in the total variation can be bounded by order of $\De t$.

In section \ref{sec:IBVP}, we discuss the initial-value problem (IBVP) for the Euler-Poisson system. For large $BV$ initial and boundary data, we prove the existence of $BV$ solution for the Euler-Poisson system. Similar to the Cauchy problem, we use the operator splitting method combined with a Glimm scheme. We apply the Glimm scheme to construct solution for IBVP of the isentropic Euler system in the spirit of Nishida-Smoller \cite{NS-bdy}. We remark that our construction of solution for gas dynamics in the Eulerian coordinate is different from the one arising in the work of Nishida-Smoller \cite{NS-bdy} done for the $p$-system. For hyperbolic systems of conservation laws with small $BV$ data, existence and stability of $BV$ solutions to IBVP have been studied in \cite{AC,Amadori,DM-1} mainly for two different types of boundary conditions (namely, characteristic and non-characteristic). By explicit examples, we demonstrate an ill-posedness result for IBVP of the isentropic Euler system (see Remark \ref{remark-counter}).\\

	\noi\textbf{Far field conditions:} Before we state our main result we need to define the following far field conditions. Throughout this article we assume the following far-field conditions:
	\begin{align}
	&\mu(x):=\left\{\begin{array}{rl}
	\mu^-&\mbox{ for }x<-L,\\
	\mu^+&\mbox{ for }x>L,
	\end{array}\right.\,
	\si(x):=\left\{\begin{array}{rl}
	\si^-&\mbox{ for }x<-L,\\
	\si^+&\mbox{ for }x>L,
	\end{array}\right.\label{eq:bdy-cond1}\\
	&\varrho_0(x):=\left\{\begin{array}{rl}
	\mu^-&\mbox{ for }x<-L,\\
	\mu^+&\mbox{ for }x>L,
	\end{array}\right.\,
	u_0(x):=\left\{\begin{array}{rl}
	u^-&\mbox{ for }x<-L,\\
	u^+&\mbox{ for }x>L,
	\end{array}\right.\label{eq:bdy-cond2}
	\end{align}
	for some $L>0$. We prove that these properties are preserved for any positive time $t\geq0$:
	\begin{equation}
	\begin{aligned}
	&\varrho(t,x)=\mu^-,\;t>0, \;x<-L(t),\;\varrho(t,x)=\mu^+,\;t>0,\;x<L(t),\\
	&u(t,x)=u^{-}(t),t>0, x<-L(t),\;u(t,x)=u^+(t),t>0,x>L(t),
	\end{aligned}
	\label{2.6}
	\end{equation}
	with $L(t)$ depending on the time $t$.
	Therefore (\ref{eqn-EP-3}) yields:
	\begin{equation}
	\Psi(t,x)=\Psi^-(t)-\int^x_{-\infty}\frac{q}{e}(\varrho(t,y)-\mu(y))dy\mbox{ for } t>0,x\in\R.
	\label{2.7}
	\end{equation}
	The electric field at $x=-\infty$, $\Psi^-(t)$ is assumed to be known. 
	As it is mentioned in \cite{P-R-V} using (\ref{2.6}) we can explicitly compute the values at infinity of $\varrho$, $u$, and $\Psi$. We have in particular:
	\begin{equation}
	\begin{cases}
	\begin{aligned}
	&
	u^{\pm}(t)=u_0^{\pm}-\frac{q}{m}\int^t_0e^{-\sigma^\pm(t-s)}\Psi^\pm(s) ds,\\
	&\Psi^+(t)=\Psi^+(0)+\int^t_0\frac{q}{e}(\mu^+u^+(s)-\mu^-u^-(s))ds+\Psi^-(t)-\Psi^-(0),\\
	&\Psi^+(0)=\Psi^-(0)-\int^{+\infty}_{-\infty}\frac{q}{e}(\varrho_0(y)-\mu(y))dy.
	\end{aligned}
	\end{cases}
	\label{2.9}
	\end{equation}
	Now we are ready to state our main result.
\begin{theorem}\label{theorem-1}
Let $T>0$ and $(\varrho_0,u_0)$ such that  $(\varrho_0,u_0)\in BV(\R,\R_+\times\R)$ and there exists $M,\bar{\varrho},\underline{\varrho}\in\R$ such that
	$$
	\abs{u_0(x)}\leq M,\,0<\underline{\varrho}\leq \varrho_0(x)\leq \bar{\varrho}\mbox{ for a.e. }x\in\R.
    $$
Furthermore $\sigma,\mu$ and $\Psi^-$ verify the following assumptions:
 \begin{itemize}
 	\item $\sigma\geq 0$, $\sigma\in BV(\R)$, $\mu \geq 0$, $\mu  \in BV(\R)$,
 	\item $\Psi^-\in BV(0,+\infty)$.
 \end{itemize} Then there exists $\gamma_0\in(1,2)$ such that the following holds: for any $\gamma\in(1,\gamma_0]$
there exists a BV weak solution $(\varrho,u,\Psi)$ of \eqref{eqn-EP-1}--\eqref{eq:initial-data} satisfying \eqref{eq:bdy-cond1}--\eqref{eq:bdy-cond2} on $[0,T]$.
\end{theorem}
\begin{remarka}
We would like to point out that Theorem \ref{theorem-1} extends the work of Poupaud, Rascle, Villa \cite{P-R-V} to the case $\gamma=1+2\e$, when $\e$ is sufficiently small. In \cite{MN,MN-1} Marcati and Natalini prove the existence of global weak solution in the framework of $L^\infty$ initial data without any smallness assumption on the initial data, however we think that the $BV$ framework is relevant if we wish to prove the uniqueness of the solutions by following the arguments developed by Bressan for the homogeneous case (see \cite{Br}) at least for small initial data. We leave this for future work. See \cite{AGG,AG,CG-1,CG-2} for related uniqueness results for hyperbolic balance laws with small BV data. 
\end{remarka}
\subsection{Comparison with the previous results}
We have already mentioned about the existence and uniqueness results of \cite{AGG,AG,MN,MN-1,Nishida,NS-bdy,Nisi-Smo}. In this subsection, we discuss them in more details and compare them with our result. 

The Euler-Poisson equations \eqref{eqn-EP-1}--\eqref{eqn-EP-3} can be seen as a hyperbolic system of balance law with non-local source term. For these type of equations, existence of BV solution in local time has been established \cite{CG-2} for initial data with small total variation. Authors \cite{CG-2} have also shown the uniqueness of BV solution satisfying a certain conditions (see \cite{CG-2}). The existence has been proved by using operator splitting method along with a fixed point argument. In comparison with \cite{CG-2}, we note that Theorem \ref{theorem-1} provides existence result for \eqref{eqn-EP-1}--\eqref{eqn-EP-3} with initial data having large total variation whereas the result of \cite{CG-2} is restricted for small BV data. We refer to \cite{AGG,AG} for existence of BV solutions to hyperbolic balance laws with local source term.

For 1-D hyperbolic systems, another way to prove existential results for large data is compensated compactness. For isentropic Euler equations, method has been studied extensively \cite{DCL,DiPerna-1,DiPerna-2,L1,L2}. In the context of Euler-Poisson system, it has been studied by \cite{MN,MN-1,ZhangB}. In \cite{MN,MN-1}, authors construct the solution by using Lax-Friedrich and Godunov scheme along with an operator splitting method. They have shown the convergence by adapting compensated compactness to Euler-Poisson system. Their results \cite{MN,MN-1} is valid for bounded initial data with compact support. We note that though the existence result has been shown for larger class than BV, it is not clear that solutions constructed in \cite{MN,MN-1,ZhangB} are unique. It is only known that the solutions constructed in \cite{MN,MN-1,ZhangB} belong to $L^\f$ space and the density is non-negative. Moreover, it does not guarantee that the solution remain in BV even for BV initial data. Similarly, it is also not clear whether density remained away from vacuum when initial density is away from vacuum. We would like to mention that the solution constructed in Theorem \ref{theorem-1} of the current article has the following properties: (i) they belong to BV, (ii) the density remains away from vacuum. 

In section \ref{sec:IBVP}, we study initial-boundary value problem for the Euler-Poisson system in the spirit of Nishida \cite{Nishida} and Nishida-Smoller \cite{NS-bdy}. For hyperbolic system of balance laws, the well-posedness for initial boundary value problem has been shown \cite{Amadori} for small initial and boundary data. Author \cite{Amadori} has proved the existence of BV solutions for two different types of boundary conditions namely, characteristic and non-characteristic conditions. On the other hand, Nishida \cite{Nishida} studied the initial boundary value problem for $p$-system with $\gamma=1$ and large BV data. An analogous result has been established in \cite{NS-bdy} for $p$-system with $\gamma=1+2\e_0$ when $\e_0>0$ is sufficiently small. In these articles \cite{Nishida,NS-bdy} it has been proved that the velocity component $u$ attains the boundary condition $u_b$ for a.e. $t>0$. In section \ref{sec:IBVP}, we prove the existence of BV solution to the initial-boundary value problem for the Euler-Poisson system with large initial and boundary data. We remark that the solutions constructed in section \ref{sec:IBVP} belong to BV space, hence they have boundary trace and we show that the momentum $m=\varrho u$ attains the boundary condition $m_b(t)$ for a.e. $t>0$. It is not clear that compensated compactness can be adapted in this set up as solutions which are constructed via compensated compactness \cite{MN,MN-1,ZhangB} fail to guarantee the boundary trace.   

To sum things up, we want to mention that our goal is to study the BV solutions in order to extend the Glimm-Bressan theory to the Euler-Poisson system. The current article is a part of the program to achieve a well-posed theory for weak solutions of Euler-Poisson system which consisting of the following properties: (i) uniqueness, (ii) BV regularity of solution from BV initial data, (iii) non-appearance of vacuum and (iv) consistency with initial-boundary value problem (which has important role in control theory). 



Rest of the article is organized as follows. In the next section we precise the numerical scheme that we are going to use. In section \ref{sec3}, we recall the basic properties of the Riemann problems and the description of the Lax curves in terms of the Riemann invariant. In section \ref{sec4}, we derive uniform estimates for approximate solution in $BV$ spaces by studying carefully the wave interaction of our scheme and we give the proof of the Theorem \ref{theorem-1}. At last, in section \ref{sec:IBVP} we study the initial boundary value problem for Euler-Poisson system \eqref{eqn-EP-1}--\eqref{eqn-EP-3}.
\section{Numerical scheme}\label{sec:scheme}
Our aim is to show the existence of weak solution via an adaptation of Glimm scheme and a splitting method. We slightly modify the approximate scheme defined in \cite{P-R-V}  to obtain a sequence of approximate solution $(\varrho_{\Delta x},u_{\Delta x})_{\Delta x>0}$. The next step will be to obtain  $BV$ bound via Glimm functional \cite{Glimm,Nisi-Smo} in order to pass to the limit when $\Delta t>0$ goes to $0$  and to recover an almost global weak solution of the Euler-Poisson system.   
\\
\\
\noi\textbf{Scheme:}
\\
Let $x_{i},t_n$ be defined as $x_{i}=i\De x$ for $i\in\mathbb{Z}$ and $t_n=n\De t$ for $n\geq0$ where $\De x>0 ,\De t>0$ are related as 
\begin{equation}
\frac{\De t}{\De x}= \la
\end{equation}
where $\la>0$ will be chosen later and will be sufficiently small in order to satisfy CFL condition. 
\begin{remarka}
		Note that once we fix $T$ and the initial data $\varrho_0,u_0$, the parameter $\lambda$ is fixed and it does not vary over time steps $t_n$. 
		 It is the main difference between our scheme and the scheme used in \cite{P-R-V}, in \cite{P-R-V} it seems important to ensure that $\lambda$ varies at each step of time in order to satisfy the CFL condition at each step. In our case the CFL condition will be satisfied at each step because the smallness condition on $\e$ depend on the time $T$.
	\end{remarka}
Let $I_{n,i}$ intervals be defined as $I_{n,i}:=(x_{i-1},x_{i+1})$ for $i\in\mathbb{Z}$ such that $n+i$ is even. Let $\varrho^\De_0, u_0^\De, \mu^\De, \sigma^\De$ be defined as 
\begin{align}
(\varrho^\De_0,u^\De_0)(x)&:=(\varrho_0(x_i),u_0(x_i))\mbox{ for }x\in I_{0,i},\label{def:approx-initial-1}\\
(\sigma^\De(x),\mu^\De(x))&:=(\sigma(x_i),\mu(x_i))\mbox{ for }x\in I_{0,i}.\label{def:approx-initial-2}
\end{align} 
Furthermore these functions respect the conditions (\ref{eq:bdy-cond1}) and (\ref{eq:bdy-cond2}).

 Next, we set $\Psi^-_n=\Psi^-(t_n)$ for $n\geq0$ and $u^{\pm}_0=u^\pm$. Then we define
\begin{align}
\Psi_{n+1}^+&=\Psi^+_{n}+\frac{q}{e}(\mu^+u^+_n-\mu^-u^-_n)\De t+\Psi^-_{n+1}-\Psi^-_{n},\label{12}\\
\Psi^+_0&=\Psi^-_0-\int\limits_{-\f}^{\f}\frac{q}{e}(\varrho^\De_0(y)-\mu^\De(y))\,dy.\label{12-1}
\end{align}
The discretization of the velocity at the infinity is given by:
\begin{align}
u^{\pm}_{n+1}=u^\pm_{n}\mbox{exp}(-\si^\pm\De t)-\frac{q}{m}\frac{1-\mbox{exp}(-\si^\pm\De t)}{\si^\pm}\Psi^\pm_{n+1}\mbox{ for }n\geq0.
\end{align}
We can bound the sequences  $(\Psi^\pm_n)_{n\in\mathbb{N}}$ and $(u^\pm_n)_{n\in\mathbb{N}}$ via the following lemma (see \cite{P-R-V}).
\begin{lemma}[\cite{P-R-V}]\label{lemma1}
	Let $\Psi^-\in L^{\f}_{loc}([0,\f))$ and $T>0$. Then there are two constants $E_T,C_T$ such that  for any $n\in\N$ with $t_n< T$ we have
	\begin{equation}
	\abs{\Psi^\pm_n}\leq E_T,\,\abs{u^\pm_n}\leq C_T.
	\end{equation}
\end{lemma}
Assume that $(\varrho_n^\Delta,u_n^\Delta)=(\varrho_{n,i},u_{n,i})_{i\in\mathbb{Z}}$ ($\varrho_n^\Delta$ and $u_n^\Delta$ are the piece-wise constant functions with value respectively $\varrho_{n,i}$ and  $u_{n,i}$ on $I_{n,i}$) is known for some $n\geq 1$ and satisfy:
\begin{align}
\varrho^\De_n(x)&\geq0,\label{16}\\
\varrho^\De_n(x)&=\left\{\begin{array}{rl}
\mu^-&\mbox{ for }x<-L-n\De x,\\
\mu^+&\mbox{ for }x>L+n\De x,
\end{array}\right.\label{17}\\
u^\De_n(x)&=\left\{\begin{array}{rl}
u_n^-&\mbox{ for }x<-L-n\De x,\\
u_n^+&\mbox{ for }x>L+n\De x.
\end{array}\right.\label{far-field-u}
\end{align}
Next, our goal is to define $(\varrho^\De_{n+1},u^\De_{n+1})$. For this purpose, we consider the homogeneous system 
\begin{align}
\frac{\pa}{\pa t}\varrho+\frac{\pa}{\pa x}(\varrho u)&=0,\label{eqn-Euler-1}\\
\frac{\pa}{\pa t}(\varrho u)+\frac{\pa}{\pa x}(\varrho u^2+\varrho^{\ga})&=0.\label{eqn-Euler-2}
\end{align}	
Let $U_l=(\varrho_l,u_l)$ and $U_r=(\varrho_r,u_r)$ with $\varrho_l>0, \varrho_r>0$. Consider the Riemann data 
\begin{equation}\label{def:Rie-data}
U_0(x):=\left\{\begin{array}{rl}
U_l\mbox{ for }x<0,\\
U_r\mbox{ for }x>0.
\end{array}\right.
\end{equation}
We denote Riemann solution $\mathcal{R}[U_l,U_r](t,x)$ as the entropy solution \cite{Lax} to the homogeneous system \eqref{eqn-Euler-1}--\eqref{eqn-Euler-2} with initial data $U_0$ as in \eqref{def:Rie-data}. Note that for $U_l,U_r$ close enough $\mathcal{R}[U_l,U_r](t,x)$ exists and no vacuum arises in the Riemann solution (when $U_l$ and $U_r$ are not necessary close, we refer to \cite{Nisi-Smo}). Now we define
\begin{equation}
(\varrho_{n+\frac{1}{2},i},u_{n+\frac{1}{2},i})=\mathcal{R}[U_{n,i-1},U_{n,i+1}](\De t,\De x\theta_{n})\mbox{ for some }\theta_n\in[-1,1].
\end{equation}
\begin{remarka}
As for the Glimm scheme, our scheme is defined modulo a sequence $\theta=(\theta_n)_{n\in\mathbb{N}}$, we will see that the numerical scheme converge for almost every sequence $\theta$.
\end{remarka}
We define now $\varrho_{n+\frac{1}{2}}^{\De,\theta}$  as $\varrho_{n+\frac{1}{2}}^{\De,\theta}(x):=\varrho_{n+\frac{1}{2},j}$ for $x\in I_{n+1,j}$ with $n+1+j$  even (we can define $u_{n+\frac{1}{2}}^{\De,\theta}$ in a similar way) and we set
\begin{equation}
\varrho^{\De,\theta}_{n+1}(x)=\varrho^{\De,\theta}_{n+\frac{1}{2}}(x)\mbox{ for }x\in \R.
\end{equation}
This definition of $\varrho^\De_n$ ensures that the conditions (\ref{16}) and (\ref{17}) are satisfied.

Next we consider a quantity $\xi_{n+1}$ defined as follows, 
\begin{equation}
\xi_{n+1}:=\left\{
\begin{array}{ll}
0&\mbox{ for }\abs{x}>L+(n+1)\De x,\\
-\frac{q}{e}((1+\varrho^\De_{n+1}(x))\ga_{n+1}-1-\mu^\De(x))&\mbox{ for }\abs{x}<L+(n+1)\De x.
\end{array}\right.
\label{23}
\end{equation}
As in \cite{P-R-V} we use the corrector term $\gamma_{n+1}$ in order to estimate in a simple way the $L^1$ norm of $\xi_n$. The term $\ga_{n+1}$ is given by
\begin{align}
\de_0&=\int\limits_{-L}^{L}(1+\varrho_0^\De(y))\,dy,\\
\de_{n+1}&=\de_n+(1+\mu^-)\De x+(1+\mu^+)\De x+\De t(\mu^-u_n^--\mu^+u_n^+),\label{25}\\
\ga_{n+1}&=\de_{n+1}\left[\int\limits_{-L-(n+1)\De x}^{L+(n+1)\De x}(1+\varrho^\De_{n+1}(y))\,dy\right]^{-1}.\label{26}
\end{align}
With the definition above, we obtain the following estimate on the $L^1$ norm of the sequence $(\xi_n)_{n\in\mathbb{N}}$. 
\begin{lemma}\label{lemma-2}
We have for $T>0$ and $\lambda\leq \frac{1}{C_T}$ with $C_T>0$ defined in Lemma \ref{lemma1}:
\begin{equation}
0\leq \de_n\leq \de_T+(2+\mu^-+\mu^+)n\De x\;\;\mbox{ for}\;t_n\leq T,
	\label{27}
	\end{equation}
	with $\de_T=\de_0+(\mu^-+\mu^+)TC_T$. Moreover,
	\begin{align}
	\int\limits_{-\f}^{\f}\xi_{n+1}(y)\,dy&=\int\limits_{-\f}^{\f}\xi_n(y)\,dy+\frac{q \De t}{e}(\mu^+u_n^+-\mu^-u_n^-),\label{28}\\
	\norm{\xi_n}_{L^1}&\leq \xi_T+4\frac{q}{e}(1+\norm{\mu^\De}_{L^\f})(L+n\De x),\mbox{ for }t_n\leq T \label{29}
	\end{align}
	where $\xi_T=\norm{\xi_0}_{L^1}+\frac{q}{e}T(\mu^-+\mu^+)C_T$.
\end{lemma}
For a detailed proof of Lemma \ref{lemma-2}, we refer to \cite{P-R-V}. To make it convenient to reader, here we briefly discuss some of the key steps.

\begin{proof}[Sketched proof of Lemma \ref{lemma-2}:]
The proof of (\ref{27}) yields from the formula (\ref{25}) and the Lemma \ref{lemma1}. In addition we use the fact that using Lemma \ref{lemma1}, we have:
$$
\begin{aligned}
&(1+\mu^-)\De x+(1+\mu^+)\De x+\De t(\mu^-u_n^--\mu^+u_n^+)\\
&=2\De x+\mu^-\De t(\frac{\De x}{\De t}+u_n^-)+\mu^+\De t(\frac{\De x}{\De t}-u_n^+)\\
&\geq 2\De x+\mu^-\De t(\frac{1}{\lambda}-C_T)+\mu^+\De t(\frac{1}{\lambda}-C_T)\\
&\geq 0.
\end{aligned}
$$
From (\ref{23}) and (\ref{26}), we deduce that:
$$
\begin{aligned}
& \int^{+\infty}_{-\infty}\xi_{n+1}(y) dy=-\frac{q}{e}\big(\delta_{n+1}-\int^{L+(n+1)\De x}_{-L-(n+1)\De x}(1+\mu^\Delta(y))dy\big).
\end{aligned}
$$
Using  (\ref{25}) and the fact that $\mu^\Delta=\mu^+$ for $y>L$ and $\mu^\Delta=\mu^-$ for $y<L$, we get (\ref{28}). Now from (\ref{23}) we have the following since $\gamma_n\geq 0$ (this is a direct consequence of the definition (\ref{26}) and the fact that $\delta_n\geq 0$):
$$
\begin{aligned}
&|\xi_{n}|\leq - \xi_{n}+\frac{2 q}{e}(1+\mu^\De).
\end{aligned}
$$
It implies that:
$$ \int^{+\infty}_{-\infty}|\xi_{n}(y) |dy\leq - \int^{+\infty}_{-\infty}\xi_{n}(y) dy+\frac{2 q}{e}\int^{L+nh}_{-L-nh}(1+\mu^\De(y))dy.$$
But from (\ref{28}) we get:
$$  \int^{+\infty}_{-\infty}\xi_n(y) dy\leq\|\xi_0\|_{L^1}+\frac{q}{e}n \De t(N^-+N^+)C_T.$$
The last two estimates lead to (\ref{29}). 
	
\end{proof}

We define $\Psi_{n}^\De:\R\rr\R$ for $n\geq0$ as follows
\begin{equation}
\Psi^\De_n(x):=\Psi_{n,i}\mbox{ for }x\in I_{n,i}\mbox{ where }\Psi_{n,i}=\Psi^{-}_{n}+\int\limits_{-\f}^{x_{i+1}}\xi_{n}(y)\,dy.
\end{equation}
From (\ref{12}), (\ref{23}) and (\ref{28}) we deduce as in \cite{P-R-V} the following Lemma.
\begin{lemma}[\cite{P-R-V}]
	For $x>L+(n+1)\De x,\,\Psi^\De_{n+1}(x)=\Psi^+_{n+1}$ and $\Psi^\De_{n+1}(x)=\Psi^-_{n+1}$ for $x<-L-(n+1)\De x$.
	\label{lem23}
\end{lemma}
Now we are ready to define $u_{n+1,i}$ as follows,
\begin{equation}\label{def:u:n+1}
u_{n+1,i}=u_{n+\frac{1}{2},i}\mbox{exp}(-\si_{i}^\De\De t)-\frac{q}{m}\frac{1-\mbox{exp}(-\si_{i}^\De\De t)}{\si^\De_{i}}\Psi_{n+1,i}.
\end{equation}
As previously we set:
$$u^{\De,\theta}_{n+1}(x)=u_{n+1,j}\;\;\mbox{on}\;I_{n+1,j}\;\mbox{with}\;n+1+j\;\mbox{even}.$$
We can note that the Lemma \ref{lem23} implies the assumption (\ref{far-field-u}). 
Finally we can define the solution $(\varrho^{\De ,\theta}, u^{\De ,\theta},\Psi^{\De ,\theta})$ of our numerical scheme for any time $t\in[0,T]$ with $T>0$ defined in the Theorem \ref{theorem-1}. First for $t_n\leq T$ we set:
\begin{equation}
\varrho^{\De ,\theta}(t_n,x)=\varrho^\De_n,\quad u^{\De,\theta}(t_n,x)=u^\De_n,
\label{6.1}
\end{equation}
\begin{equation}
(\varrho^{\De ,\theta},u^{\De,\theta})\;\mbox{satisfies (\ref{eqn-Euler-1})-(\ref{eqn-Euler-2})}\;\mbox{on}\;(t_n,t_{n+1}).
\label{6.2}
\end{equation}
Indeed we can solve (\ref{eqn-Euler-1})-(\ref{eqn-Euler-2}) on $(t_n,t_{n+1})$ at the condition that there is no interaction between the solutions of each Riemann problem (we will see later that it will be the case by imposing a CFL condition on the parameter $\lambda$).

The electric field is now defined by:
\begin{equation}
\begin{cases}
\begin{aligned}
&\gamma^{\De,\theta}(t)=\gamma_n,\;t_n\leq t<t_{n+1},\\
&\xi^{\De,\theta}=-\frac{q}{e}((1+\varrho^{\De,\theta})\gamma^{\De,\theta}-1-\mu^\De),\;x\in\R,\;t>0,\\
&\Psi^{\De,\theta}(t,x)=\Psi^-(t)+\int^x_{-\infty}\xi^{\De,\theta}(y)dy, \;x\in\R,\;t\in[0,T].
\end{aligned}
\end{cases}
\label{6.5}
\end{equation}

\section{Riemann Invariant and Lax curves}
\label{sec3}
In this section, we discuss some basic facts on the homogeneous problem for the Euler system and recall some results on Riemann invariant from Nishida-Smoller \cite{Nisi-Smo}. We can rewrite the system \eqref{eqn-Euler-1}--\eqref{eqn-Euler-2} in density and momentum variables $\varrho,m$ as follows
\begin{align}
\frac{\pa}{\pa t}\varrho+\frac{\pa}{\pa x}m&=0,\label{eqn-homo-1}\\
\frac{\pa}{\pa t}m+\frac{\pa}{\pa x}\left(\frac{m^2}{\varrho}+\varrho^{\ga}\right)&=0.\label{eqn-homo-2}
\end{align}	
This system is hyperbolic for $\varrho>0$ and eigenvalues are
\begin{equation}
\begin{aligned}
&\la_1(\varrho,u)=\frac{m}{\varrho}-\sqrt{\gamma}\varrho^{\frac{\ga-1}{2}}=u-\sqrt{\gamma}\varrho^{\frac{\ga-1}{2}},
\\
&\la_2(\varrho,u)=\frac{m}{\varrho}+\sqrt{\gamma}\varrho^{\frac{\ga-1}{2}}=u+\sqrt{\gamma}\varrho^{\frac{\ga-1}{2}}.
\end{aligned}
\end{equation}
\begin{enumerate}
	\item \textbf{Shock waves:} We denote the admissible shock curve via the Lax criterion $S_i(\varrho_-,m_-)$ starting from $(\varrho_-,m_-)\in(0,\f)\times\R$ defined as 
	\begin{equation}\label{shock-curve-m}
	m-m_-=\frac{m_-}{\varrho_-}(\varrho-\varrho_-)+(-1)^{i}(\varrho-\varrho_-)\sqrt{\frac{\varrho(\varrho^{\ga}-\varrho_-^\ga)}{\varrho_-(\varrho-\varrho_-)}}\mbox{ if }(-1)^i(\varrho-\varrho_-)<0\mbox{ for }i=1,2.
	\end{equation}
	In $(\varrho,u)$ variables, we can rewrite \eqref{shock-curve-m} as
	\begin{equation}\label{shock-curve-u}
	u-u_-=-\sqrt{\frac{(\varrho^{\ga}-\varrho_-^\ga)(\varrho-\varrho_-)}{\varrho_-\varrho}}\mbox{ if }(-1)^i(\varrho-\varrho_-)<0\mbox{ for }i=1,2.
	\end{equation}
	\item \textbf{Rarefaction waves:} We denote rarefaction curve $R_i(\varrho_-,m_-)$ starting from $(\varrho_-,m_-)\in(0,\f)\times\R$ defined as 
	\begin{equation}\label{rare-curve-m}
	m-m_-=\frac{m_-}{\varrho_-}(\varrho-\varrho_-)+(-1)^{i}\frac{2\sqrt{\gamma}}{\gamma-1}\varrho\left(\varrho^{\frac{\ga-1}{2}}-\varrho_-^{\frac{\ga-1}{2}}\right)\mbox{ if }(-1)^i(\varrho-\varrho_-)>0\mbox{ for }i=1,2.
	\end{equation}
	In $(\varrho,u)$ variables, we can rewrite \eqref{rare-curve-m} as
	\begin{equation}\label{rare-curve-u}
	u-u_-=(-1)^{i}\frac{2\sqrt{\gamma}}{\gamma-1}\left(\varrho^{\frac{\ga-1}{2}}-\varrho_-^{\frac{\ga-1}{2}}\right)\mbox{ if }(-1)^i(\varrho-\varrho_-)>0\mbox{ for }i=1,2.
	\end{equation}

\end{enumerate}

\noi\textbf{Riemann invariant:} The corresponding Riemann invariants $r,s$ are as follows
\begin{equation}
r(\varrho,u)=u-\sqrt{\ga}\frac{\varrho^\e-1}{\e}\mbox{ and }s(\varrho,u)=u+\sqrt{\ga}\frac{\varrho^\e-1}{\e}\mbox{ where }\e=\frac{\ga-1}{2}.
\end{equation}
We recall the following result on the Riemann problem solved by Riemann \cite{Rie}.
\begin{lemma}\label{lemma:Euler-1}
	The Cauchy problem \eqref{eqn-homo-1}--\eqref{eqn-homo-2} with Riemann data \eqref{def:Rie-data} has a piece-wise continuous solution in $\R\times\R_+$ satisfying
	\begin{equation}
	r(t,x)=r(\varrho(t,x),u(t,x))\geq\min\{r_-,r_+\},\,s(x,t)=s(\varrho(t,x),u(t,x))\leq\max\{s_-,s_+\}
	\end{equation}
	where $r_\pm=r(\varrho_\pm,u_\pm),s_\pm=s(\varrho_\pm,u_\pm)$ with $s_--r_+>-\frac{\sqrt{\ga}}{\e}$.
\end{lemma}
\begin{remarka}
The condition $s_--r_+>-2\sqrt{\ga}/\e$ is equivalent to $u_+-u_-<\frac{\sqrt{\gamma}}{\e}(\varrho_-^\e+\varrho_+^\e)$, in particular for any state $U_-$, $U_+$ we can solve the Riemann problem provided that $\e>0$ is sufficiently small.
\end{remarka}
In terms of Riemann coordinate $(r,s)$ we rephrase the shock curves $S_i\equiv S_i(r_-,s_-)=S_i(\varrho_-,u_-)$ for $i=1,2$ as below (see \cite{Nisi-Smo})
\begin{align}
&S_1:\left\{\begin{array}{rl}
r_--r&=\varrho_-^\e\left[\sqrt{\frac{(\al-1)(\al^{\ga}-1)}{\al}}+\sqrt{\ga}\frac{\al^{\e}-1}{\e}\right],\\
s_--s&=\varrho_-^\e\left[\sqrt{\frac{(\al-1)(\al^{\ga}-1)}{\al}}-\sqrt{\ga}\frac{\al^{\e}-1}{\e}\right],
\end{array}\right.\mbox{ where }\al=\frac{\varrho}{\varrho_-}\geq 1,\\
&S_2:\left\{\begin{array}{rl}
s_--s&=\varrho_-^\e\left[\sqrt{\frac{(1-\al)(1-\al^{\ga})}{\al}}+\sqrt{\ga}\frac{1-\al^{\e}}{\e}\right],\\
r_--r&=\varrho_-^\e\left[\sqrt{\frac{(1-\al)(1-\al^{\ga})}{\al}}-\sqrt{\ga}\frac{1-\al^{\e}}{\e}\right],
\end{array}\right.\mbox{ where }0<\al=\frac{\varrho}{\varrho_-}\leq 1.\label{shock-curve-2}
\end{align}
Borrowing notation from \cite{Nisi-Smo} we write respectively for the $S_1$ and $S_2$ admissible shocks $s_--s=g_1(r_--r,\varrho_-)$ for $r\leq r_-$ and $r_--r=g_2(s_--s,\varrho_-)$ for $s\leq s_-$.
\begin{lemma}[Nishida-Smoller, \cite{Nisi-Smo}]\label{lemma:NS}
	Let $g_i(\al,\varrho_-)$ be defined above for $i=1,2,\varrho_->0$. Then
	\begin{equation}
	0\leq g^\p_i(\al,\varrho_-)<1\mbox{ and }g^{\p\p}_i(\al,\varrho_-)\geq0\mbox{ for }i=1,2,
	\label{47}
	\end{equation}
	where we use the notation $g^{\p}_i=\frac{\pa g_i}{\pa \al},g^{\p\p}_i=\frac{\pa^2 g_i}{\pa^2 \al}$ for $i=1,2$.
\end{lemma}
\begin{remarka}
Now we know that for a shock wave  with left state $(r_-,s_-)$  and right state $(r_+,s_+)$, we have for example if we consider a 1 shock
$|r_--r_+|=|\gamma|$ and $|s_--s_+|\leq |\gamma|$ using (\ref{47}). In particular it implies that:
$$|u_--u_+|\leq \frac{1}{2}(|r_--r_+|+|s_--s_+|)\leq |\gamma|.$$
We have a similar result for a 2 shock wave.\label{remimp}
\end{remarka}

We denote $\Omega_{I}[r_0,s_0],\Omega_{II}[r_0,s_0],\Omega_{III}[r_0,s_0],\Omega_{IV}[r_0,s_0]$ as follows (see Figure \ref{fig-1b} for a demonstration)
\begin{align}
\Omega_{I}[r_0,s_0]&:=\left\{(r,s);\,r\leq r_0,s\leq s_0\mbox{ and }r_0-g_2(s_0-s,\varrho_0)\geq r,s_0-g_1(r_0-r,\varrho_0)\geq s\right\},\nonumber\\
\Omega_{II}[r_0,s_0]&:=\left\{(r,s);\,r\leq r_0,s\leq s_0\mbox{ and }r_0-g_2(s_0-s,\varrho_0)\leq r\right\}\cup\left\{(r,s);\,r\geq r_0,s\leq s_0 \right\},\nonumber\\
\Omega_{III}[r_0,s_0]&:=\left\{(r,s);\,r\geq r_0,s\geq s_0\right\},\nonumber\\
\Omega_{IV}[r_0,s_0]&:=\left\{(r,s);\,r\leq r_0,s\leq s_0\mbox{ and }s_0-g_1(r_0-r,\varrho_0)\leq s\right\}\cup\left\{(r,s);\,r\leq r_0,s\geq s_0 \right\}.\nonumber
\end{align}

	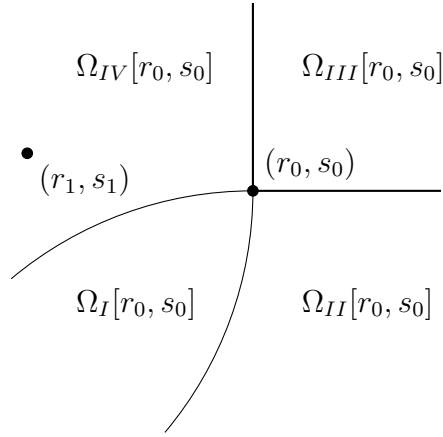
\begin{figure}[ht]
	\centering
	\begin{tikzpicture}
	\draw (2.5,4) arc (90:130:5cm);
	\draw (2.5,4) arc (0:-40:5cm);
	
	\draw[thick] (2.5,4)--(2.5,6.5);
	\draw[thick] (2.5,4)--(5,4);
  	\filldraw (-0.5,4.5) circle (2pt);
  	\filldraw (2.5,4) circle (2pt);
	\draw[thick] (3,6) node[anchor=north west] {$\Omega_{III}[r_0,s_0]$};
	\draw[thick] (3,2.5) node[anchor=west] {$\Omega_{II}[r_0,s_0]$};
	\draw[thick] (0,6) node[anchor=north west] {$\Omega_{IV}[r_0,s_0]$};
	\draw[thick] (0,2.5) node[anchor=west] {$\Omega_{I}[r_0,s_0]$};
	\draw[thick] (2.5,4) node[anchor=south west] {$(r_0,s_0)$};
	\draw[thick] (-0.5,4.5) node[anchor=north west] {$(r_1,s_1)$};	

	\end{tikzpicture}
	
	\caption{For a point $(r_0,s_0)$, the $r$--$s$ plane is divided into four disjoint subsets, $\Omega_{I}[r_0,s_0],\,\Omega_{II}[r_0,s_0],\,\Omega_{III}[r_0,s_0]$ and $\Omega_{IV}[r_0,s_0]$. Here the point $(r_1,s_1)$ belongs to $\Omega_{IV}[r_0,s_0]$. For a Riemann data $U_l,U_r$ with $U_l=(r_0,s_0),\,U_r=(r_1,s_1)$, the admissible solution consists 1-rarefaction and 2-shock. } \label{fig-1b}
\end{figure}
\begin{remarka} It is important to note (see \cite{Smoller}) that if $(r_1,s_1)\in \Omega_{I}[r_0,s_0]$ and that $(r_0,s_0)$, $(r_1,s_1)$ satisfy the assumptions of Lemma \ref{lemma:Euler-1}, then the Riemann problem issue of the left state $(r_0,s_0)$ and right state $(r_1,s_1)$ is described by three different state which are relied by a 1 shock and a 2 shock. Similarly if  $(r_1,s_1)\in \Omega_{II}[r_0,s_0]$, then the solution of the Riemann problem is the composition of 1 rarefaction and a 2 shock. We have similar properties when $(r_1,s_1)$ is respectively in  $\Omega_{III}[r_0,s_0]$ or
$\Omega_{IV}[r_0,s_0]$.
\end{remarka}
 We give now a useful Lemma for the sequel which allows to identify the position of translated states in the different regions $\Omega_I,\cdots, \Omega_{IV}$. 
\begin{lemma}\label{lemma:position}
		Let $(r_0,s_0)\in\R^2$ be a point in $r$-$s$ plane. Suppose $(r_0,s_0)$ corresponds to $(\varrho_0,u_0)$. Consider a point $(r_1,s_1)\in\Omega_{I}[r_0,s_0]$ corresponds to $(\varrho_1,u_1)$. Let $\de_0,\de_1>0$ and $\bar{u}_0,\bar{u}_1$ be defined as $\bar{u}_i=u_i+\de_i$ for $i=0,1$. Then we have
		\begin{enumerate}
			\item If $(r_1,s_1)\in\Omega_{I}[r_0,s_0]$, then $(r_1,s_1)\in\Omega_{I}[\bar{r}_0,\bar{s}_0]$. Furthermore, if $\delta_0=\delta_1$, then we have $(\bar{r}_1,\bar{s}_1)\in\Omega_{I}[\bar{r}_0,\bar{s}_0]$.
			
			\item If $(r_1,s_1)\in\Omega_{II}[r_0,s_0]$, then $(\bar{r}_1,\bar{s}_1)\notin\Omega_{I}[r_0,s_0]\cup\Omega_{IV}[r_0,s_0]$ and $(r_1,s_1)\notin\Omega_{III}[\bar{r}_0,\bar{s}_0]\cup \Omega_{IV}[\bar{r}_0,\bar{s}_0]$.
			
			\item If $(r_1,s_1)\in\Omega_{III}[r_0,s_0]$, then $(\bar{r}_1,\bar{s}_1)\in\Omega_{III}[r_0,s_0]$.
			
			\item If $(r_1,s_1)\in\Omega_{IV}[r_0,s_0]$, then $(\bar{r}_1,\bar{s}_1)\notin\Omega_{I}[r_0,s_0]\cup\Omega_{II}[r_0,s_0]$ and $(r_1,s_1)\notin\Omega_{II}[\bar{r}_0,\bar{s}_0]\cup \Omega_{III}[\bar{r}_0,\bar{s}_0]$.
			
		\end{enumerate}
\end{lemma}
 \begin{proof} We just prove here the first case. Assume that $(r_1,s_1)\in \Omega_{I}[r_0,s_0]$, it means in particular that:
 $$r_1\leq r_0\;\;\mbox{and}\;\;r_1\leq r_0-g_2(s_0-s_1,\varrho_0).$$
 From the definition of $\bar{r}_0$, $r_1\leq \bar{r}_0$. Now we write,
 $$r_1\leq r_0+\delta_0-g_2(s_0+\delta_0-s_1,\varrho_0)+(g_2(s_0+\delta_0-s_1,\varrho_0)-g_2(s_0-s_1,\varrho_0)-\delta_0).$$
 Using the fact that $0\leq g'_2(\cdot,\varrho_0)<1$ we deduce that $g_2(s_0+\delta_0-s_1,\varrho_0)-g_2(s_0-s_1,\varrho_0)-\delta_0\leq 0$ and that:
 $$r_1\leq r_0+\delta_0-g_2(s_0+\delta_0-s_1,\varrho_0).$$
%
%

 In a similar way we prove that $s_1\leq \bar{s}_0$ and $s_1\leq\bar{s}_0-g_1(\bar{r}_0-r_1,\varrho_0)$. 
 Similarly when $\delta_0=\delta_1$, if $(r_1,s_1)\in\Omega_I[r_0,s_0]$ then we have:
 $$r_1\leq r_0\;\;\mbox{and}\;\;r_1\leq r_0-g_2(s_0-s_1,\varrho_0).$$
 We deduce using the fact that $\bar{\varrho}_0=\varrho_0$:
 $$r_1+\delta_1\leq r_0+\delta_0\;\;\mbox{and}\;\;r_1+\delta_1\leq r_0+\delta_0-g_2(s_0+\delta_0-s_1-\delta_1,\bar{\varrho}_0).$$
 Similarly we have $\bar{s}_1\leq\bar{s}_0-g_1(\bar{r}_0-\bar{r}_1,\bar{\varrho}_0)$.
 The rest of the proof follows the same line.
 \end{proof}

	Now we recall the following result from Nishida-Smoller, \cite{Nisi-Smo}. This plays a crucial role in controlling Glimm functional for large data.  
\begin{lemma}[Nishida-Smoller, \cite{Nisi-Smo}]
	Let $0\leq \e<1/2,s_->s'_-$ and $\varrho_-,\varrho_-^\p\in[\underline{\varrho},\overline{\varrho}]$ with $0<\underline{\varrho}<\overline{\varrho}<\f$. Suppose  $(r_+,s_+)$ and $(r_+,s_+^\p)$ are two points on $S_1$ curves originating from $(r_-,s_-)$ $(r_-,s_-^\p)$ respectively.  
	Then we get
	\begin{equation}\label{interaction-est-1}
	0\leq (s_-^\p-s_+^\p)-(s_--s_+)\leq C\e(s_--s_-^\p)(r_--r_+)
	\end{equation}
	where $C$ is independent of $\e,\varrho_-,\varrho_-^\p$ and depending on $\underline{\varrho},\overline{\varrho}$. Similar results are true for $S_2$ curve as well.
\end{lemma}
In the sequel we denote by $S_1^\p$ and $S_2^\p$ the non physical inverse shock wave curves, $S_2^\p$ consists of those states $(r,s)$ which can be connected to the state $(r_0,s_0)$ on the right by an $S_2$ shock. In particular as previously (see \cite{Nisi-Smo}), we can represent $S_2^\p$ by the following equation:
\begin{equation}
r-r_0=g_1(s-s_0,\varrho_0)\;\;\mbox{for}\;s>s_0.
\label{S2non}
\end{equation}
Similarly we have for   $S_1^\p$:
\begin{equation}
s-s_0=g_2(r-r_0,\varrho_0)\;\;\mbox{for}\;r>r_0.
\label{S1non}
\end{equation}
\section{Proof of the Theorem \ref{theorem-1}}\label{sec4}
\subsection{New estimates on the wave interactions}\label{sec:new-est}

\begin{lemma}\label{lemma:estimate-1}
	Let $(r_0,s_0)\in\R^2$ be a point in $r$-$s$ plane. Consider a point $(r_1,s_1)\in\Omega_{I}[r_0,s_0]$. Let $\de>0$ and $(r_2,s_2)$ be defined as $(r_2,s_2)=(r_0+\de,s_0+\de)$. Let $\B+\gamma$ be the outgoing wave from Riemann data $U_l=(r_0,s_0)$ and $U_r=(r_1,s_1)$ and $\B^\p+\gamma^\p$ be the outgoing wave from Riemann data $U_l=(r_2,s_2)$ and $U_r=(r_1,s_1)$. Then we have
	\begin{equation}
	\abs{\B^\p}\leq \abs{\B}+\de\;\;\mbox{and}\;\;\abs{\gamma^\p}\leq \abs{\gamma}+\de.
	\end{equation} 
\end{lemma}
 \begin{remarka}
Due to $2\sqrt{\gamma}\frac{\varrho^\e-1}{\e}=s-r$ we have in the previous Lemma that $\varrho_2=\varrho_0$.
\end{remarka}

   \begin{proof}[Proof of Lemma \ref{lemma:estimate-1}:]
   Using the Lemma \ref{lemma:position}, we know that the Riemann problems associated to the states  $[(r_0,s_0),(r_1,s_1)]$ and $[(r_2,s_2),(r_1,s_1)]$ are solved via a 1 shock and a 2 shock.
	Let the $S_2^\p$ curve from $(r_1,s_1)$ and the $S_1$ curve from $(r_0,s_0)$ intersect at $(r_3,s_3)$. Let the $S_2^\p$ curve from $(r_1,s_1)$ and the $S_1$ curve from $(r_2,s_2)$ intersect at $(r_4,s_4)$. Note that $(r_3,s_3),(r_4,s_4)$ lie on the curve $r-r_1=g_1(s-s_1,\varrho_1)$. Suppose the line $L_1=\{(r_3+d,s_3+d);\,d\geq0\}$ intersects $S_1$ curve from $(r_2,s_2)$ at $(\bar{r},\bar{s})$. As $S_1$ curve from $(r_2,s_2)$ is represented by $s-s_2=g_1(r-r_2,\varrho_0)$ we have  for $r\leq \bar{r}$ since $g^{\p}_1<1$:
$$
\begin{aligned}
s-s_2&=\bar{s}-s_2-\int^{\bar{r}-r_2}_{r-r_2}g'_1(\beta,\varrho_0) d\beta,\\
&s-\bar{s}\geq r-\bar{r}.
\end{aligned}
$$
It implies that the curve $s-s_2=g_1(r-r_2,\varrho_0)$ lies in $\{(r,s);s-\bar{s}>r-\bar{r},r\leq \bar{r},s\leq\bar{s}\}$ .  Now if $s_4<s_3$ it implies necessary that $r_4\leq r_3$ since $(r_4,s_4)$ and $(r_3,s_3)$ are on the curve $r-r_1=g_1(s-s_1,\varrho_1)$ with $s\geq s_1$ and $g'_1(\cdot,\varrho_1)\geq 0$, we can prove again since $g'_1<1$ that $(r_3,s_3)$ lies in $\{(r,s);s-s_4>r-r_4, s\geq s_4, r\geq r_4\}$ and it is not possible since $(r_4,s_4)$ is in  $\{(r,s);s-\bar{s}>r-\bar{r},r\leq\bar{r}, s\leq\bar{s}\}$. Therefore we have $s_4>s_3$. Since $0\leq g_1^\p$ and $(r_3,s_3)$, $(r_4,s_4)$ are on the curve $S_2^\p$ and $s_4>s_3$, we have $r_3\leq r_4$ (see Figure \ref{fig-1a} for clear illustration). Hence we have:
$$|\beta'|=r_2-r_4=r_0-r_3+\delta+r_3-r_4\leq |\beta|+\delta.$$
  Next, we observe that $\bar{r}=r_3+\de,\bar{s}=s_3+\de$ since $S_1$ curve remains unchanged under translation along $\varrho=\varrho_0$ line (and we have $\varrho_2=\varrho_0$). Since $(r_4,s_4)\in\{(r,s);s-\bar{s}>r-\bar{r},r\leq \bar{r},s\leq\bar{s}\}$ we have $s_4\leq \bar{s}=s_3+\de$. Therefore, we obtain
	$$
	\abs{\gamma^\p}\leq \abs{\gamma}+\de.
	$$
\end{proof}
\begin{figure}[ht]
		\centering
		\begin{tikzpicture}
		\draw (-1.2,1.5) arc (180:150:9cm);
		
		\draw (2.5,4) arc (90:110:12.5cm);
		\draw (3.5,6) arc (90:115:12cm);
		
		\draw[thick] (2.5,4)--(3.5,6);
		\draw[thick] (-0.99,3.5)--(.59,6.55);
		
		\draw[thick,color=blue] (-0.99,3.5)--(-0.99,2.7);
		
		\draw[dashed,<->] (-0.99,3.3)--(2.5,3.3);
		\draw (1,3.3) node[fill=white, inner sep=1pt](N){$\abs{\B}$};	
		
		\draw[dashed,<->] (2.5,3)--(3.5,3);
		\draw (3,3) node[fill=white, inner sep=1pt](N){\small$\de$};	
		
		\draw[dashed,<->] (-0.32,7)--(3.5,7);
		\draw (1.5,7) node[fill=white, inner sep=1pt](N){$\abs{\B^\p}$};	
		
		\draw[dashed,<->] (-3.5,3.5)--(-3.5,5.5);
		\draw (-3.5,4.5) node[fill=white, inner sep=1pt](N){$\de$};	
		
		\draw[dashed,<->] (-2,1.5)--(-2,3.5);
		\draw (-2,2.5) node[fill=white, inner sep=1pt](N){$\abs{\gamma}$};	
		
     	\draw[dashed,<->] (-2.8,1.5)--(-2.8,5.38);
        \draw (-2.8,4.5) node[fill=white, inner sep=1pt](N){$\abs{\gamma^\p}$};	
		
		\draw[thick,color=blue] (0.075,5.5)--(0.075,2.7);
		
		\draw[thick,color=blue] (2.5,4)--(2.5,2.7);
		\draw[thick,color=blue] (3.5,7.5)--(3.5,6)--(3.5,2.7);
		\draw[thick,color=blue] (-0.99,3.5)--(-4,3.5);
		\draw[thick,color=blue] (-1.2,1.5)--(-4,1.5);
		
		\draw[thick,color=blue] (2.5,4)--(4,4);
		\draw[thick,color=blue] (3.5,6)--(4,6);
		
		\draw[dashed,<->] (3.8,4)--(3.8,6);
		\draw (3.8,5) node[fill=white, inner sep=1pt](N){$\de$};	
		
			\draw[thick] (2.4,3.8) node[anchor=west] {\tiny$(r_0,s_0)$};
				\draw[thick] (3.5,6.2) node[anchor=west] {\tiny$(r_2,s_2)$};
				\draw[thick] (-1.2,1.5) node[anchor=west] {\tiny$(r_1,s_1)$};
				
				\draw[thick] (-0.83,3.45) node[anchor=south east] {\tiny$(r_3,s_3)$};
				\draw[thick] (0.07,5.65) node[anchor=north west] {\tiny$(\bar{r},\bar{s})$};


		\draw[thick,color=blue] (-0.32,7.5)--(-0.32,5.38)--(-4,5.38);
		\draw[thick,color=blue] (0.1,5.5)--(-4,5.5);
		\draw (-0.5,5.3) node[fill=white, inner sep=0.3pt,anchor=north east] {\tiny$({r}_4,{s}_4)$};
		\end{tikzpicture}
		
		\caption{The $S_1$ curve starting from $(r_0,s_0)$ intersects the $S^\p_2$ curve from $(r_1,s_1)$ at the point $(r_3,s_3)$. Also, the $S_1$ curve starting from $(r_2,s_2)$ intersects the $S^\p_2$ curve from $(r_1,s_1)$ at the point $(r_4,s_4)$. Here, $r_0-r_3=\abs{\B},\,r_2-r_4=\abs{\B^\p}$ and $s_3-s_1=\abs{\ga},\,s_4-s_1=\abs{\ga^\p}$. In this figure, $(\bar{r},\bar{s})=(r_3+\de,s_3+\de)$.} \label{fig-1a}
	\end{figure}
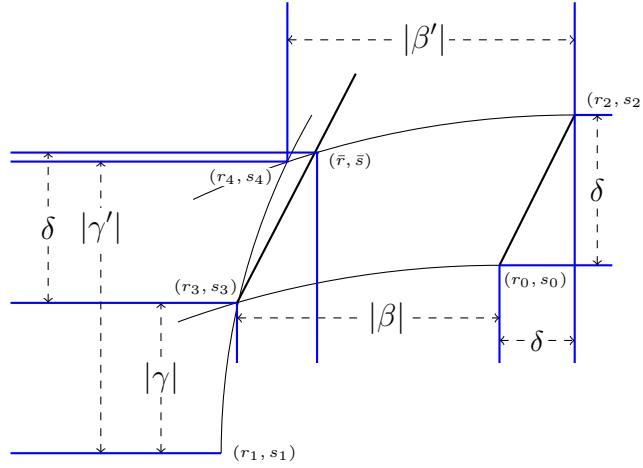
	
By a similar argument we get 
\begin{lemma}\label{lemma:estimate-2}
	Let $(r_0,s_0)\in\R^2$ be a point in $r$-$s$ plane. Consider a point $(r_1,s_1)\in\Omega_{I}[r_0,s_0]$. Let $\de>0$ and $(r_2,s_2)$ be defined as $(r_2,s_2)=(r_1+\de,s_1+\de)$ with $\delta>0$ such that $(r_2,s_2)\in\Omega_{I}[(r_0,s_0)]$. Let $\B+\gamma$ be the outgoing wave from Riemann data $U_l=(r_0,s_0)$ and $U_r=(r_1,s_1)$ and $\B^\p+\gamma^\p$ be the outgoing wave from Riemann data $U_l=(r_0,s_0)$ and $U_r=(r_2,s_2)$. Then we have
	\begin{equation}
	\abs{\B^\p}\leq \abs{\B}+\de\;\;\mbox{and}\;\;\abs{\gamma^\p}\leq \abs{\gamma}+\de.
	\end{equation} 
\end{lemma}
\begin{lemma}\label{lemma:estimate-3}
Let $(\delta_-,\delta_+)\in\mathbb{R}^2$
and $(r_+,s_+),(r_-,s_-)\in\R^2$  satisfying the assumptions of Lemma \ref{lemma:Euler-1} such that we can get a Riemann problem solution with $U_l=(r_-,s_-)=(\varrho_-,u_-),U_r=(r_+,s_+)=(\varrho_+,u_+)$. Let $u^\p_\pm=u_\pm+\de_\pm$ and $(r_\pm^\p,s_\pm^\p)$ be Riemann invariant corresponding to $(\varrho_\pm,u^\p_\pm)$. We assume again that we can solve the Riemann problem associated to $(r_\pm^\p,s_\pm^\p)$. Then we have the following. 
	\begin{enumerate}[label=\arabic*.]
		\item When $(r_+^\p,s_+^\p)\in\Omega_{I}[r_-^\p,s_-^\p]$,
		\begin{enumerate}[label=1.\arabic*]
			\item\label{1.1} If $(r_+,s_+)\in\Omega_{I}[r_-,s_-]$ then 
			\begin{equation*}
			\abs{\B^\p}\leq \abs{\B}+\abs{\de_+-\de_-}\;\;\mbox{and}\;\;\abs{\gamma^\p}\leq \abs{\gamma}+\abs{\de_+-\de_-}.
			\end{equation*}
			\item\label{1.4} If $(r_+,s_+)\in\Omega_{IV}[r_-,s_-]$ then 
			\begin{equation*}
			\abs{\B^\p}\leq \abs{\B}+\abs{\de_+-\de_-}\;\;\mbox{and}\;\;\abs{\gamma^\p}\leq \abs{\de_+-\de_-}.
			\end{equation*}
			\item If $(r_+,s_+)\in\Omega_{II}[r_-,s_-]$ then 
			\begin{equation*}
			\abs{\gamma^\p}\leq \abs{\gamma}+\abs{\de_+-\de_-}\;\;\mbox{and}\;\;\abs{\B^\p}\leq \abs{\de_+-\de_-}.
			\end{equation*}
			\item If $(r_+,s_+)\in\Omega_{III}[r_-,s_-]$ then 
			\begin{equation*}
			\abs{\B^\p}+\abs{\gamma^\p}\leq 2\abs{\de_+-\de_-}.
			\end{equation*}

		\end{enumerate}

  	\item When $(r_+^\p,s_+^\p)\in\Omega_{II}[r_-^\p,s_-^\p]$,
  \begin{enumerate}[label=2.\arabic*]
  	\item If $(r_+,s_+)\in\Omega_{I}[r_-,s_-]$ then 
  	\begin{equation*}
  	\abs{\gamma^\p}\leq \abs{\B}+\abs{\gamma}+\abs{\de_+-\de_-}.
  	\end{equation*}
  	
  	\item If $(r_+,s_+)\in\Omega_{II}[r_-,s_-]$ then 
  	\begin{equation*}
  	\abs{\gamma^\p}\leq \abs{\gamma}+\abs{\de_+-\de_-}.
  	\end{equation*}
  	\item If $(r_+,s_+)\in\Omega_{III}[r_-,s_-]$ then 
  	\begin{equation*}
  	\abs{\gamma^\p}\leq \abs{\de_+-\de_-}.
  	\end{equation*}
  	
  \end{enumerate}
      	\item When $(r_+^\p,s_+^\p)\in\Omega_{IV}[r_-^\p,s_-^\p]$,
\begin{enumerate}[label=3.\arabic*]
	\item If $(r_+,s_+)\in\Omega_{I}[r_-,s_-]$ then 
	\begin{equation*}
	\abs{\B^\p}\leq \abs{\B}+\abs{\gamma}-\abs{\de_+-\de_-}.
	\end{equation*}
	
	\item If $(r_+,s_+)\in\Omega_{III}[r_-,s_-]$ then 
	\begin{equation*}
	\abs{\B^\p}\leq \abs{\de_+-\de_-}.
	\end{equation*}
	\item If $(r_+,s_+)\in\Omega_{IV}[r_-,s_-]$ then 
	\begin{equation*}
	\abs{\B^\p}= \abs{\B}+\abs{\de_+-\de_-}.
	\end{equation*}
	
\end{enumerate}

	\end{enumerate}
\end{lemma}

\begin{proof}
	We first prove for $(r_+^\p,s_+^\p)\in\Omega_{I}[r_-^\p,s_-^\p]$ and split the proof in four sub-cases depending on the fact that Riemann problem associated to $(r_-,s_-)$ and $(r_+,s_+)$ is solved by combining 1 shock or 1 rarefaction and 2 shock or 2 rarefaction. In order to prove these different cases, we will use a generic notation $(\tilde{r}_-,\tilde{s}_-),(\bar{r}_+,\bar{s}_+)$ for the following.
	\begin{enumerate}
		\item When $\de_-\geq \de_+$ we shift $(r_+^\p,s_+^\p)$ to $(r_+,s_+)$ and $(r_-^\p,s_-^\p)$ is shifted to the point $(\tilde{r}_-,\tilde{s}_-):=(r_-^\p-\de_+,s_-^\p-\de_+)=(r_-+\de_1,s_-+\de_1)$ where $\de_1=\de_--\de_+\geq0$. 
		
		\item When $\de_+> \de_-$ we shift $(r_-^\p,s_-^\p)$ to $(r_-,s_-)$ and $(r_+^\p,s_+^\p)$ is shifted to the point $(\bar{r}_+,\bar{s}_+):=(r_+^\p-\de_-,s_+^\p-\de_-)=(r_++\de_2,s_++\de_2)$ where $\de_2=\de_+-\de_->0$. 
	\end{enumerate}
	
	\begin{enumerate}[label=Case-1.\arabic*]
		\item	We first consider the case when $(r_+,s_+)\in \Omega_{I}[r_-,s_-]$. Then there arise two possibilities.
		\begin{enumerate}
			\item $\de_-\geq\de_+$. From Lemma \ref{lemma:position} and since $\delta_1\geq 0$ we know that  $(r_+,s_+)\in \Omega_{I}(\tilde{r}_-,\tilde{s}_-)$. We suppose now that the $S_1$ curves from $(r_-,s_-),(\tilde{r}_-,\tilde{s}_-)$ intersect with $S_2^\p$ curve from $(r_+,s_+)$ at $(r_1,s_1),(r_2,s_2)$ respectively. Then by Lemma \ref{lemma:estimate-1} we get $\abs{r_2-\tilde{r}_-}\leq \abs{r_1-r_-}+\de_1$ and $\abs{s_2-\tilde{s}_-}\leq \abs{s_1-s_-}+\de_1$. This gives us the required estimate because we have $(\tilde{r}_-,\tilde{s}_-):=(r_-^\p-\de_+,s_-^\p-\de_+)$ and $(r_+,s_+)=(r_+^\p-\de_+,s_+^\p-\de_+)$. For a clear illustration of this case see Figure \ref{fig-1}.
			
			\item Now since $(\bar{r}_+,\bar{s}_+)=(r_+^\p-\de_-,s_+^\p-\de_-)$, $(r_-,s_-)=(r_-^\p-\de_-,s_-^\p-\de_-)$ and $(r_+^\p,s_+^\p)\in\Omega_{I}[r_-^\p,s_-^\p]$ we deduce from Lemma \ref{lemma:position} that $(\bar{r}_+,\bar{s}_+)\in\Omega_I(r_-,s_-)$. Now consider $\de_+>\de_-$. This case can be handled similarly as in case $\de_-\geq\de_+$ and we use Lemma \ref{lemma:estimate-2} instead of Lemma \ref{lemma:estimate-1} to get the required estimate.
		\end{enumerate}
	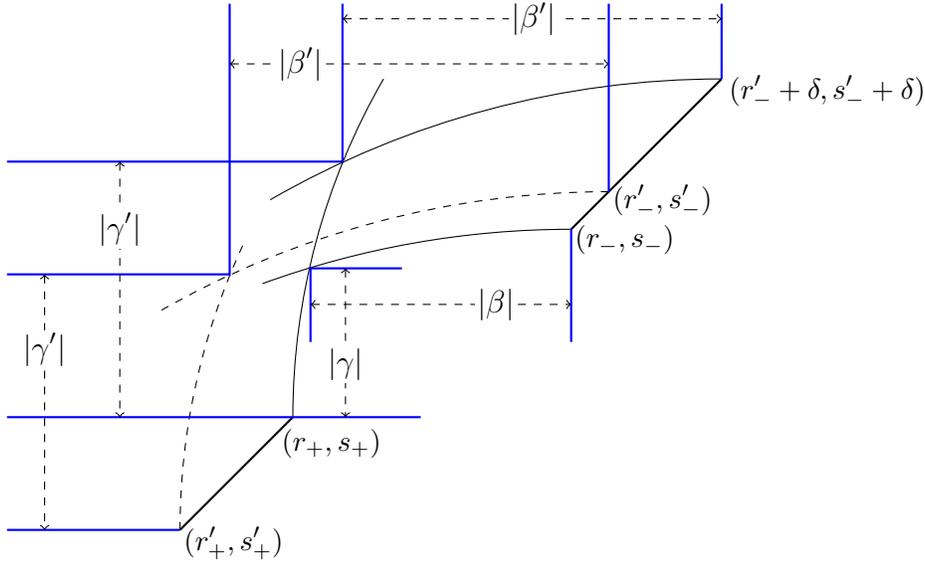
\begin{figure}[ht]
		\centering
		\begin{tikzpicture}
		\draw (-1.2,1.5) arc (180:150:9cm);
		
		\draw (2.5,4) arc (90:110:12cm);
		\draw (4.5,6) arc (90:120:12cm);
		
		\draw[dashed] (3,4.5) arc (90:120:12cm);
		\draw[dashed] (-2.7,0) arc (180:155:9cm);
		
		\draw[thick] (2.5,4)--(4.5,6);
		\draw[thick] (-2.7,0)--(-1.2,1.5);
		
        \draw[thick, color=blue](-5,4.9)--(-0.54,4.9)--(-0.54,7);
        \draw[thick, color=blue](-5,3.4)--(-2.04,3.4)--(-2.04,7);
        
        \draw[thick, color=blue](-5,1.5)--(0.5,1.5);
        \draw[thick, color=blue](-5,0)--(-2.7,0);
        
        \draw[thick, color=blue](3,4.5)--(3,7);
        \draw[thick, color=blue](2.5,4)--(2.5,2.5);
        \draw[thick, color=blue](4.5,6)--(4.5,7);
        
        \draw[thick, color=blue](-0.965,2.5)--(-0.965,3.48)--(0.25,3.48);
        
    	\draw[dashed,<->] (-0.965,3)--(2.5,3);
        \draw (1.5,3) node[fill=white, inner sep=1pt](N){$\abs{\B}$};
        
        \draw[dashed,<->] (-0.5,3.48)--(-0.5,1.5);
        \draw (-0.5,2.2) node[fill=white, inner sep=1pt](N){$\abs{\gamma}$};
        
        \draw[dashed,<->] (-3.5,4.9)--(-3.5,1.5);
        \draw (-3.5,4) node[fill=white, inner sep=1pt](N){$\abs{\gamma^\p}$};
        
        \draw[dashed,<->] (-4.5,3.4)--(-4.5,0);
        \draw (-4.5,2.4) node[fill=white, inner sep=1pt](N){$\abs{\gamma^\p}$};	
        
        \draw[dashed,<->] (-0.54,6.8)--(4.5,6.8);
        \draw (2,6.8) node[fill=white, inner sep=1pt](N){$\abs{\B^\p}$};
        
        \draw[dashed,<->] (-2.04,6.2)--(3,6.2);
        \draw (-1.1,6.2) node[fill=white, inner sep=1pt](N){$\abs{\B^\p}$};
        
       	\draw[thick] (2.4,4.2) node[anchor=north west] {\small$(r_-,s_-)$};
       	
       	\draw[thick] (2.9,4.75) node[anchor=north west] {\small$(r_-^\p,s_-^\p)$};
       	
       	\draw[thick] (4.45,6.2) node[anchor=north west] {\small$(r_-^\p+\de,s_-^\p+\de)$};
       	
       	\draw[thick] (-1.5,1.5) node[anchor=north west] {\small$(r_+,s_+)$};
       	
       	\draw[thick] (-2.8,0.2) node[anchor=north west] {\small$(r_+^\p,s_+^\p)$};

		\end{tikzpicture}
		
		\caption{From the Riemann data $(r_-,s_-)$ and $(r_+,s_+)$ two shocks of strengths $\abs{\B},\abs{\ga}$ arise. Two shocks situation arises even for Riemann data  $(r^\p_-,s^\p_-)$ and $(r^\p_+,s^\p_+)$ where $r^\p_\pm=r_\pm+\de_\pm,\,s^\p_\pm=s_\pm+\de_\pm $ with $\de_->0>\de_+$. In the later case, shock strengths are $\abs{\B^\p},\abs{\ga^\p}$ which remains same even after a translation $(r,s)\mapsto (r+\de,s+\de)$ where $\de=\abs{\de_+}$.}
		\label{fig-1}
	\end{figure}

\item Now we consider the case when $(r_+,s_+)\in \Omega_{IV}[r_-,s_-]$. Then there arise two possibilities.

\begin{enumerate}
	\item $\de_-\geq \de_+$. Since $(r_+^{\p},s_+^\p)\in \Omega_I[r_-^\p,s_-^\p]$ and  $(\tilde{r}_-,\tilde{s}_-)=(r^\p_--\delta_+,s_-^\p-\delta_+)$, $(r_+,s_+)=(r^\p_+-\delta_+,s_+^\p-\delta_+)$ we deduce from Lemma \ref{lemma:position} that $(r_+,s_+)\in\Omega_I[\tilde{r}_-,\tilde{s}_-]$. Since $\de_1:=\de_--\de_+\geq 0$, the point $(r_+,s_+)$ lies between two $S_1$ curves starting from $(r_-,s_-)$ and $(\tilde{r}_-,\tilde{s}_-)$. Hence there exists a $\eta\in(0,\de_1)$ such that we can shift the point $(r_-,s_-)$ to $(r_1,s_1)=(r_-+\eta,s_-+\eta)$ such that $(r_+,s_+)$ lies on the $S_1$ curve of $(r_1,s_1)$. It follows of the transversality of the $S_1$ curves with the segment $[(r_-,s_-),(\tilde{r}_-,\tilde{s}_-)]$. Let $\B_0$ denotes the shock connecting $(r_1,s_1)$ and $(r_+,s_+)$. Then by Lemma \ref{lemma:estimate-1} we obtain $\abs{\gamma^\p}\leq (\de_1-\eta)$
and	$\abs{\B^\p}\leq \abs{\B_0}+(\de_1-\eta)$. Since $\abs{\B_0}=\abs{\B}+\eta$ we get $ \abs{\beta^\p}\leq \abs{\B}+\de_1$. This case is demonstrated in Figure \ref{fig-2}.
	
	\item  $\de_+> \de_-$. Note that by 4) Lemma \ref{lemma:position}, we get $(\bar{r}_+,\bar{s}_+)\notin\Omega_{I}[{r}_-,{s}_-]$. This gives a contradiction since we know that  $(\bar{r}_+,\bar{s}_+)\in\Omega_{I}[{r}_-,{s}_-]$. It comes from the 1) Lemma \ref{lemma:position} and the fact that
$(\bar{r}_+,\bar{s}_+)=(r^\p_+,s^\p_+)-(\de_-,\de_-)$	and $(r_-,s_-)=(r^\p_-,s^\p_-)-(\de_-,\de_-)$ with $(r^\p_+,s^\p_+)\in\Omega_I[r^\p_-,s^\p_-]$.
	
\end{enumerate}

	\begin{figure}[ht]
	\centering
	\begin{tikzpicture}
	\draw (-0.8,4) arc (180:140:3cm);
	
	\draw (1.5,3) arc (90:100:25cm);
	
	\draw[dashed] (2.75,4.25) arc (90:103:25cm);
	
	\draw (4.5,6) arc (90:105:25cm);
	
	
	\draw[thick] (1.5,3)--(4.5,6);
	\draw[thick] (-0.8,4)--(-0.8,2.9);
	
	\filldraw (-0.8,4) circle (2pt);
	\filldraw (1.5,3) circle (2pt);
	\filldraw (4.5,6) circle (2pt);
	\filldraw (2.75,4.25) circle (2pt);

%
%
	\draw[thick, color=blue](-0.8,2.9)--(-0.8,2);
	\draw[thick, color=blue](1.5,3)--(1.5,2);
	\draw[thick, color=blue] (4.5,6)--(4.5,2);
	\draw[thick, color=blue](2.75,4.25)--(2.75,2);
	\draw[thick, color=blue](-0.39,5.52)--(-0.39,4.5);
	\draw[thick, color=blue](-0.39,5.52)--(-4,5.52);
	\draw[thick, color=blue](-0.8,4)--(-4,4);
%
%
	\draw[dashed,<->] (-0.8,2.3)--(1.5,2.3);
	\draw (.3,2.3) node[fill=white, inner sep=1pt](N){$\abs{\B}$};
	\draw[dashed,<->] (1.5,2.75)--(2.75,2.75);
	\draw (2.2,2.75) node[fill=white, inner sep=1pt](N){$\eta$};
	\draw[dashed,<->] (1.5,2.5)--(4.5,2.5);
	\draw (3.5,2.5) node[fill=white, inner sep=1pt](N){$\de$};
	
	\draw[dashed,<->] (-3,5.52)--(-3,4);
	\draw (-3,4.7) node[fill=white, inner sep=1pt](N){$\abs{\gamma^\p}$};	
	
	\draw[dashed,<->] (-0.39,5)--(4.5,5);
	\draw (2.5,5) node[fill=white, inner sep=1pt](N){$\abs{\B^\p}$};
		
	\draw[thick] (-0.7,3.9) node[anchor=north east] {\small$(r_+,s_+)$};
	
	\draw[thick] (1.76,3) node[anchor=south east] {\small$(r_-,s_-)$};
	
	\draw[thick] (2.84,4.2) node[fill=white, inner sep=1pt, anchor=north west] {\small$(r_-+\eta,s_-+\eta)$};
  	
  	\draw[thick] (4.5,6) node[anchor=west] {\small$(r_-^\p,s_-^\p)$};
	
	\end{tikzpicture}
	
	\caption{Two points $(r_-,s_-),(r_+,s_+)$ are related as $(r_+,s_+)\in\Omega_{IV}[r_-,s_-]$ which means the Riemann data $(r_-,s_-),(r_+,s_+)$ gives 1-shock of strength $\abs{\B}$ and 2-rarefaction. Under the translation $(r,s)\mapsto(r+\de,s+\de)$, the point $(r_-,s_-)$ goes to $(r^\p_-,s^\p_-)$. For the Riemann data $(r^\p_-,s^\p_-),(r_+,s_+)$, two shocks arise with strengths $\abs{\B^\p},\abs{\ga^\p}$.  } 
\label{fig-2}
\end{figure}

\item The case when $(r_+,s_+)\in\Omega_{II}[r_-,s_-]$ can be handled in the similar way.

\item Suppose $(r_+,s_+)\in\Omega_{III}[r_-,s_-]$. 
\begin{enumerate}
	\item Suppose $\de_-\geq\de_+$. Since we have seen that $(r_+,s_+)\in\Omega_{I}[\tilde{r}_-,\tilde{s}_-]$, we have $0\leq \tilde{r}_--r_+\leq \de_1$ and $0\leq \tilde{s}_--s_+\leq\de_1$. Hence we get the required estimate. See Figure \ref{fig-3} for a clear illustration of this case.
	
	\item  Consider $\de_+>\de_-$. By applying Lemma \ref{lemma:position} we get $(\bar{r}_+,\bar{s}_+)\notin\Omega_{I}[r_-,s_-]$ which is contradiction to our assumption $({r}_+^\p,{s}_+^\p)\in \Omega_{I}[r^\p_-,s^\p_-]$.
\end{enumerate}

	\begin{figure}[ht]
	\centering
	\begin{tikzpicture}
	\draw (-0.8,4) arc (180:140:3cm);
	
	
	
	\draw (2.5,6) arc (90:110:12cm);
	
	
	\draw[thick] (2.5,6)--(-2,1.5);
	\draw[thick] (-0.8,4)--(-0.8,1.5)--(-2,1.5);
	
	\filldraw (-0.8,4) circle (2pt);
	\filldraw (2.5,6) circle (2pt);
	\filldraw (-2,1.5) circle (2pt);

	\draw[thick, color=blue](-3,5.68)--(-0.29,5.68)--(-0.29,6.8);
	\draw[thick, color=blue] (2.5,6)--(2.5,6.8);
	\draw[thick, color=blue](-0.8,4) --(-3,4);
	\draw[thick, color=blue](-0.8,1.5)--(3.8,1.5);
	\draw[thick, color=blue] (-2,1.5)--(-2,3.5);
	\draw[thick, color=blue](2.5,6)--(2.5,2);
	\draw[thick, color=blue](2.5,6)--(3.8,6);	
	\draw[dashed,<->] (3.1,6)--(3.1,1.5);
	\draw (3.1,4) node[fill=white, inner sep=1pt](N){$\de$};
	
	\draw[dashed,<->] (-2,3)--(2.5,3);
	\draw (0.5,3) node[fill=white, inner sep=1pt](N){$\de$};
	
	\draw[dashed,<->] (-0.29,6.5)--(2.5,6.5);
	\draw (1.2,6.5) node[fill=white, inner sep=1pt](N){$\abs{\B}$};
	
	\draw[dashed,<->] (-2.5,4)--(-2.5,5.68);
	\draw (-2.5,4.8) node[fill=white, inner sep=1pt](N){$\abs{\gamma}$};	
	
	\draw (-2,1.5) node[anchor= south east]{$(r_-,s_-)$};
		\draw (-0.83,4) node[anchor= south west]{$(r_+,s_+)$};
			\draw (2.5,6) node[anchor= south west]{$(r^\p_-,s^\p_-)$};

	\end{tikzpicture}
	
	\caption{Three  points $(r_-,s_-),(r_+,s_+),(r^\p_-,s^\p_-)$ are considered in $r$-$s$ plane such that $r_-^\p=r_-+\de,\,s^\p_-=s_-+\de$. From Riemann data $(r_-,s_-),(r_+,s_+)$ two rarefaction waves arise whereas the Riemann data $(r^\p_-,s^\p_-),(r_+,s_+)$ gives two shocks of strengths $\abs{\B},\abs{\ga}$. } \label{fig-3}
\end{figure}
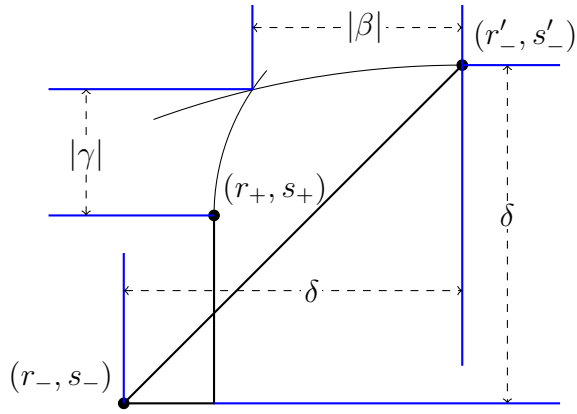
\end{enumerate}
Next we prove the case $(r_+^\p,s_+^\p)\in\Omega_{II}[r_-^\p,s_-^\p]$. Proof of this case follows in a similar way as in the previous case. For proof of this case, we use a generic notation $(\tilde{r}_-,\tilde{s}_-),(\bar{r}_+,\bar{s}_+)$ for the following.
\begin{enumerate}
	\item When $\de_-\geq \de_+$ we shift $(r_+^\p,s_+^\p)$ to $(r_+,s_+)$ and $(r_-^\p,s_-^\p)$ is shifted to the point $(\tilde{r}_-,\tilde{s}_-):=(r_-^\p-\de_+,s_-^\p-\de_+)=(r_-+\de_1,s_-+\de_1)$ where $\de_1=\de_--\de_+\geq0$. 
	
	\item When $\de_+> \de_-$ we shift $(r_-^\p,s_-^\p)$ to $(r_-,s_-)$ and $(r_+^\p,s_+^\p)$ is shifted to the point $(\bar{r}_+,\bar{s}_+):=(r_+^\p-\de_-,s_+^\p-\de_-)=(r_++\de_2,s_++\de_2)$ where $\de_2=\de_+-\de_->0$. 
\end{enumerate}

	\begin{enumerate}[label=Case-2.\arabic*]
	\item Consider the case when $(r_+,s_+)\in\Omega_{I}[r_-,s_-]$. As before, we divide into two cases: (i) $\de_-\geq \de_+$, (ii) $\de_+>\de_-$. 
	\begin{enumerate}
		\item $\de_-\geq\de_+$. In this case, by Lemma \ref{lemma:position} we note that $(r_+,s_+)\notin \Omega_{II}[\tilde{r}_-,\tilde{s}_-]$.
		
		\item $\de_+>\de_-$. Since $(\bar{r}_+,\bar{s}_+)=(r_++\de_2,s_++\de_2)\in \Omega_{II}[r_-,s_-]$ and $(r_+,s_+)\in\Omega_{I}[r_-,s_-]$, we can find a $\de\in[0,\de_2]$ such that $(r_-,s_-)$ lies on the $S_2^\p$ curve starting from $(r_++\de,s_++\de)$. By Lemma \ref{lemma:estimate-2} we get $\abs{\gamma^\p}\leq \de+\abs{\B}+\abs{\ga}$. Hence, we prove the estimate $	\abs{\gamma^\p}\leq \abs{\B}+\abs{\gamma}+\abs{\de_+-\de_-}$. 
	\end{enumerate}

     \item Consider the case when $(r_+,s_+)\in\Omega_{II}[r_-,s_-]$. As before, we divide into two cases: (i) $\de_-\geq \de_+$, (ii) $\de_+>\de_-$. 
     \begin{enumerate}
     	\item $\de_-\geq \de_+$. In this case we note that $\tilde{s}_--s_+=s_--s_++\de_1$ where $\de_1=\de_--\de_+$. Hence we have $\abs{\gamma^\p}=\abs{\gamma}+\de_1$.
     	
     	\item $\de_+>\de_-$. Observe that $\bar{s}_--s_+=s_--s_+-\de_2$ where $\de_2=\de_+-\de_-$. Hence we have $\abs{\gamma^\p}=\abs{\gamma}-\de_2$.

     \end{enumerate}
 
 \item Consider the case when $(r_+,s_+)\in\Omega_{IV}[r_-,s_-]$. As before, we divide into two cases: (i) $\de_-\geq \de_+$, (ii) $\de_+>\de_-$. 
 \begin{enumerate}
 	\item $\de_-\geq \de_+$. In this case, by invoking Lemma \ref{lemma:position} we get $(r_+,s_+)\notin\Omega_{II}[\tilde{r}_-,\tilde{s}_-]$. It is a contradiction. 
 	
 	\item $\de_+>\de_-$. Again by Lemma \ref{lemma:position} we get $(\bar{r}_+,\bar{s}_+)\notin\Omega_{II}[{r}_-,{s}_-]$. This gives a contradiction.
 	
 \end{enumerate}

 \item Consider the case when $(r_+,s_+)\in\Omega_{III}[r_-,s_-]$. As before, we divide into two cases: (i) $\de_-\geq \de_+$, (ii) $\de_+>\de_-$. 
 \begin{enumerate}
 	\item $\de_-\geq \de_+$. Since $(r_+,s_+)\in\Omega_{III}[r_-,s_-]$, we have $s_-\leq s_+$. Hence $\tilde{s}_--s_+=s_-+\de_1-s_+$ and this gives the required estimate.
 	
 	\item $\de_+>\de_-$. In this case we apply Lemma \ref{lemma:position} to get $(\bar{r}_+,\bar{s}_+)\notin\Omega_{II}[r_-,s_-]$ which gives a contradiction.
 	
 \end{enumerate}
 
\end{enumerate}
This completes the proof for case $(r_+^\p,s_+^\p)\in\Omega_{II}[r_-^\p,s_-^\p]$. By a similar argument we can show the case $(r_+^\p,s_+^\p)\in\Omega_{IV}[r_-^\p,s_-^\p]$. This ends the proof of Lemma \ref{lemma:estimate-3}.

\end{proof}

\subsection{Estimates of Glimm functional}\label{sec:Glimm}
\begin{definition}
	\begin{enumerate}
		\item An  $I$-curve is a piece-wise linear, Lipschitz continuous function such that each linear part coincides either with the line joining $(n\De t,x_i+\theta_n\De x), ((n+1)\De t,x_i+\De x+\theta_{n+1}\De x), $  or with the line joining $(k\De t, x_j+\theta_k\De x), ((k+1)\De t,x_j-\De x+\theta_{k+1}\De x)$.  
		\item For $n\geq0$, we define $O^n$-curve is an $I$-curve contained in $\{(t,x); x\in\R,n\De t\leq t\leq (n+1)\De t\}$,
		 we simply write $O$-curve instead of $O^0$.
	\end{enumerate}
\end{definition}

\noi\textbf{Glimm functional:} For an $I$-curve $J$ we define
\begin{align}
V(J)&=\sum\left\{\abs{\al}:\,\al \mbox{ is a shock wave crossing }J\right\},\label{def:V}\\
Q(J)&=\sum\left\{\abs{\beta}\abs{\gamma}: \beta,\gamma\;\mbox{ cross $J$ and approach}\right\},\label{def:Q}\\
F(J)&=V(J)+KQ(J),\label{def:F}
\end{align}
with $K>0$ which will be defined later. 
\begin{remarka}
It is important to point out that for a $J$ curve there is no shock wave crossing $J$ at infinity since the solution of the Riemann problem on time interval $(t_n,t_{n+1})$ is trivial since the solution $(\varrho^\De,u^\De)$ is constant at infinity at the time $t_n$.
\end{remarka}
\begin{remarka}
As in \cite{Nisi-Smo} we can estimate the total variation of the solution $(\varrho^\De,u^\De)$ along $O^n$
 by using the quantity $V(O^n)$. 
When a shock wave cross $O^n$ \footnote{There is no shock wave no rarefaction wave crossing $O^n$ when $O^n$ is includes in the region $]-\infty,-L-(n+1)\De t[$ and $]L+(n+1)\De t,+\infty[$ since the solution $(\varrho^{\De,\theta},u^{\De,\theta})$ is constant and takes respectively the values $(\varrho_-,u_n^-)$ and $(\varrho_+,u_n^+)$ on $[t_n,t_{n+1})$.}, the Riemann invariants $r$ and $s$ decrease whereas $r$ and $s$ increase when a rarefaction wave cross $O^n$. We deduce using Remark \ref{remimp} that the decreasing total variation on $r$ and $s$ is controlled by $V(O^n)$, in particular if we consider the restriction of the functions $r^{\De,\theta}$, $s^{\De,\theta}$ on $O^n$ that we note $r^{\De,\theta}_{|O^n}$, $s^{\De,\theta}_{|O^n}$ where we assume that $r^{\De,\theta}_{|O^n}$ and  $s^{\De,\theta}_{|O^n}$ take the same values as $r^{\De,\theta}$ and $s^{\De,\theta}$ except on the points $(t_{n+1},x_k+\theta_{n+1})$ where $r^{\De,\theta}_{|O^n}$ and  $s^{\De,\theta}_{|O^n}$ take the values $r_{n+\frac{1}{2},k}$, $s_{n+\frac{1}{2},k}$. It implies that the total variation of  $r^{\De,\theta}_{|O^n}$, $s^{\De,\theta}_{|O^n}$ along $O^n$ is bounded by:
\begin{equation}
\begin{aligned}
&TV(r^{\De,\theta}_{|O^n})
&\leq 2 V(O^n)+ |r_n^{+}-r_n^-|,
\end{aligned}
\end{equation}
\begin{equation}
\begin{aligned}
&TV(s^{\De,\theta}_{|O^n})\leq 2 V(O^n)+|s_n^{+}-s_n^-|.
\end{aligned}
\end{equation}
It implies in particular that we have with our definition of $r^{\De,\theta}_{|O^n}$ and $s^{\De,\theta}_{|O^n}$ that:
\begin{equation}
\begin{aligned}
&TV(r^{\De,\theta}(t_n,\cdot))\leq 2 V(O^n)+ |r_n^{+}-r_n^-|,\; TV(s^{\De,\theta}(t_n,\cdot))\leq 2 V(O^n)+ |r_n^{+}-r_n^-
\\
&TV(r_{n+\frac{1}{2}}^{\De,\theta}(t_n,\cdot))\leq 2 V(O^n)+ |r_n^{+}-r_n^-|,\;TV(s_{n+\frac{1}{2}}^{\De,\theta}(t_n,\cdot))\leq 2 V(O^n)+ |r_n^{+}-r_n^-|.
\end{aligned}
\end{equation}
Using Lemma \ref{lemma1} and the definition of the Riemann invariant we deduce that:
\begin{equation}
\begin{aligned}
&TV(u^{\De,\theta}(t_n,\cdot))\leq 4 V(O^n)+C'_T,\\
&TV(u^{\De,\theta}(t_{n+\frac{1}{2}},\cdot))\leq 4 V(O^n)+C'_T,
\label{techimp}
\end{aligned}
\end{equation}
with $C'_T$ a positive constant depending only on $T$.
\label{remimp1a}
\end{remarka}
We recall the following Lemma for the $O$ curve (see \cite{Nisi-Smo} p197).
\begin{lemma}\label{lemma-esti-1}
	Let $\e_1$ such that $4C\e_1TV(r_0(\cdot),s_0(\cdot)) \leq1$. Let $0\leq \e\leq \min \{\e_0,\e_1\}$ where $\e_0$ is as in Lemma \ref{lemma:Euler-1} for Riemann problems issue of the values defined by $(\varrho_0,u_0)$ with $0<\underline{\varrho}\leq \varrho_0(x)\leq \bar{\varrho}<\f$ for $x\in\R$. Suppose $F$ is defined as in \eqref{def:F} for $K=4C\e$ where $C$ is the constant as in \eqref{interaction-est-1}.Then we have
	\begin{equation}
	F(O)=V(O)+KQ(O)\leq 2V(O)\leq  2 TV(r_0(\cdot),s_0(\cdot)).
	\end{equation}
\end{lemma}

\begin{lemma}
	Let $T>0$. Then for $n\geq1$ such that $t_n\leq T$, we have
	\begin{align}
	TV(\Psi_n^{\De,\theta})&\leq \xi_T+4\frac{q}{e}\left(1+\norm{\mu^\De}_{L^\f}\right),\label{TV-Psi}\\
		\norm{\Psi_{n}^{\De,\theta}}_{L^\f(\R)}&\leq \norm{\Psi^-}_{L^\f(0,T)}+\xi_T+4\frac{q}{e}\left(1+\norm{\mu^\De}_{L^\f}\right),\label{L-infty-Psi}\\
	\norm{\varrho_n^{\De,\theta}}_{L^\f(\R)}&\leq C(\rho^-,V(O^n)),\label{L-infty-varrho}\\
	\norm{u_n^{\De,\theta}}_{L^\f(\R)}&\leq C_T+ 4V(O^n),\label{L-infty-u}
	\end{align}
		where $C(\cdot,\cdot)$ depends only on $\underline{\varrho},\bar{\varrho},M$ and $\xi_T=\norm{\xi_0}_{L^1}+T(\mu^-+\mu^+)C_T$ with $C_T$ is the same constant as in Lemma \ref{lemma1}.
\end{lemma}
\begin{proof}
	From the definition of $\Psi^\De_n$ and Lemma \ref{lemma-2} we have
	\begin{equation*}
	TV(\Psi^\De_n)=\sum\limits_{n+j\in2\mathbb{Z}}\abs{\Psi_{n,j}-\Psi_{n,j+2}}\leq \int\limits_{\R}\abs{\xi_n^\De}\,dx\leq \xi_T+4\frac{q}{e}\left(1+\norm{\mu^\De}_{L^\f}\right)
	\end{equation*}
		where $\xi_T=\norm{\xi_0}_{L^1}+T(\mu^-+\mu^+)C_T$. To get \eqref{L-infty-Psi} note the following
		\begin{align}
		\norm{\Psi_{n}^\De}_{L^\f(\R)}\leq \Psi_n^-+TV(\Psi^\De_n)&\leq \norm{\Psi^-}_{L^\f(0,T)}+TV(\Psi^\De_n)\nonumber\\
		&\leq \norm{\Psi^-}_{L^\f(0,T)}+\xi_T+4\frac{q}{e}\left(1+\norm{\mu^\De}_{L^\f}\right).\label{ineq-L-infty-Psi}
		\end{align}
The estimate (\ref{L-infty-u}) is a direct consequence of the Remark  \ref{remimp1a} and of the Lemma \ref{lemma1}. Similarly since we can write $\varrho$ in terms of $r(\varrho,u)$ and $s(\varrho,u)$,
and using the fact that the Riemann invariant defined $C^1$ diffeomorphism with respect to $\varrho,u$ we obtain (\ref{L-infty-Psi})
	with $C(\cdot,\cdot)$ depending only on $\underline{\varrho},\bar{\varrho},M$. Now with a similar argument as in \eqref{ineq-L-infty-Psi} we get \eqref{L-infty-varrho}, \eqref{L-infty-u}.
\end{proof}
We Suppose $A_1,B_1$ be two constants such that for $L_1=2L+\frac{2T}{\lambda}$
\begin{align}
A_1&\geq 4e^{\la A  L_1}TV(\si^\De )+16 \norm{\si^\De}_\f,\  \label{def:A1}\\
	B_1&\geq 4C_TTV(\si^\De )+4B\frac{e^{A\la L_1}-1}{A}TV(\si^\De )+4 C'_T  \norm{\si^\De}_\f+\frac{4q}{m}\norm{\Psi_{n+1}^\De}_\f TV(\si^\De)\De t\label{def:B1} \\
	&+\frac{4q}{m}TV(\Psi^\De_{n+1}),\nonumber
\end{align}
where $A,B>0$ are defined as 
\begin{equation}\label{def:A-B}
A\geq 32\norm{\si^\De}_\f\mbox{ and }B\geq 8(C_T+C'_T)\norm{\si^\De }_{\f}+\frac{8q}{m}\left[\norm{\Psi^\De}_{\f}\abs{\si^\De}_\f\De t+\norm{\Psi^\De}_\f\right].
\end{equation}
\begin{lemma}\label{lemma:BV-bound}
	Let $A_1,B_1>0$ be defined as in \eqref{def:A1} and \eqref{def:B1}. Let $\e_2>0$ satisfies
	\begin{equation}
	4C\e_2\left[e^{\la L_1}e^{A_1n\De t}TV(r_0,s_0)+B\frac{e^{A\la L_1}-1}{A}e^{A_1n\De t}+B_1\frac{e^{A_1n\De t}-1}{A_1}\right]\leq \min\{1,C_0\},
	\end{equation} 
	where $C>0$ is a constant appeared in \eqref{interaction-est-1} and $C_0>0$ is the constant as in \cite[Lemma 4, pgae-193]{Nisi-Smo}. Let $\e_0,\e_1$ be as in Lemma \ref{lemma-esti-1}. Then for $0\leq \e\leq \min\{\e_i,i=0,1,2\}$ we have
	\begin{equation}\label{estimate-FOn}
	F(O^{n})\leq e^{A_1n\De t}F(O)+B_1\sum\limits_{k=0}^{n-1}(1+A_1\De t)^k\mbox{ for }0\leq n+1\leq \frac{T}{\De t}. 
	\end{equation}
\end{lemma}
\begin{remark}
	In approximation process, we choose $\si^\De,\mu_-^\De$ such that $\norm{\si^\De}_\f\leq \norm{\si}_\f$, $TV(\si^\De)\leq TV(\si)$ and $\norm{\mu_-^\De}_\f\leq \norm{\mu_-}_\f$. By \eqref{TV-Psi} and \eqref{L-infty-Psi}, we get that the constants $A_1,B_1$ are independent of mesh size.
\end{remark}
\begin{proof}
	We prove \eqref{estimate-FOn} for $0\leq n\leq T/\De t$ by induction. For $n=0$, the estimate \eqref{estimate-FOn} trivially follows. We assume that the inequality \eqref{estimate-FOn} is true for $n$ then we prove it for $n+1$. By \eqref{far-field-u} we know that for a sufficiently large $i_n\geq1$, 
	\begin{align*}
	(\varrho_{n+1,i},u_{n+1,i})&=(\varrho_{n,i+1},u_{n,i+1})=(\varrho_-,u_{n+1}^-)\mbox{ for }i\De x\leq -L-(n+1)\De x,\\
	(\varrho_{n+1,i},u_{n+1,i})&=(\varrho_{n,i-1},u_{n,i-1})=(\varrho_+,u_{n+1}^+)\mbox{ for }i\De x\geq L+(n+1)\De x.
	\end{align*}
	Therefore, we can reach $O^{n+1}$ from $O^n$ in finitely many steps by considering consecutively immediate successor (indeed there is nothing to do on $O^{n+1}$ when $x\leq -L-(n+1)\De x$ and when $x\geq L+(n+1)\De x$ since we know that there is no shock wave crossing $O^{n+1}$), that is, there are $\bar{J}_i,0\leq i\leq m$ such that $O^n=\bar{J}_0\leq\bar{J}_1\leq\cdots\leq \bar{J}_m=O^{n+1}$ where $\bar{J}_{i+1}$ is immediate successor of $\bar{J}_i$ for $i\geq0$. We also observe that $m\De x\leq L_1=2L+T$. 
	 \begin{claim}\label{claim-1}
	 	Let $A,B>0$ be defined as in \eqref{def:A-B}. Then we have,
	 	\begin{equation}\label{estimate-FJi}
	 	F(\bar{J}_i)\leq e^{\la Ai\De x}F(O^n)+B\De t\sum\limits_{l=0}^{i-1}(1+A\De t)^l\mbox{ for }0\leq i\leq m.
	\end{equation}
	 \end{claim}
 \begin{proof}[Proof of Claim \ref{claim-1}] We will also prove this claim by induction. Note that $i=0$ case is trivial. Next we assume that \eqref{estimate-FJi} is true for $i=j$, then we show for $i=j+1$. 
 	
 	First, we assume that $J_2$ is an immediate successor to $J_1$ and we wish to evaluate $F(J_2)$ in terms of $F(J_1)$.
	Let $J_1\setminus J_2$ be consisting lines  $\mathcal{L}[(t_{n+1},x_{k}+\theta_{n+1}\De x),(t_{n},x_{k+1}+\theta_{n}\De x)]$ and $\mathcal{L}[(t_{n+1},x_{k+2}+\theta_{n+1}\De x),(t_{n},x_{k+1}+\theta_{n}\De x)]$. Let $J_2\setminus J_1$ be consisting lines  $\mathcal{L}[(t_{n+1},x_{k}+\theta_{n+1}\De x),(t_{n+2},x_{k+1}+\theta_{n+2}\De x)]$ and $\mathcal{L}[(t_{n+1},x_{k+2}+\theta_{n+1}\De x),(t_{n+2},x_{k+1}+\theta_{n+2}\De x)]$. We can observe that we define a diamond-shaped region $D_{n+1,k}$ with vertices at the four surrounding sampling points $(t_{n+1},x_{k}+\theta_{n+1}\De x)$, $(t_{n},x_{k+1}+\theta_{n}\De x)$, $(t_{n+1},x_{k+2}+\theta_{n+1}\De x)$ and $(t_{n+2},x_{k+1}+\theta_{n+2}\De x)$ with respectively the following state $(\varrho_{n+1,k},u_{n+1,k})$, $(\varrho_{n,k+1},u_{n,k+1})$, $(\varrho_{n+1,k+2},u_{n+1,k+2})$ and $(\varrho_{n+2,k+1},u_{n+2,k+1})$ (see Figure \ref{fig-4} for an illustration).
	\begin{figure}[ht]
		\centering
		\begin{tikzpicture}
		
		\draw[thick][color=black] (-5 ,0) -- (4,0) ;
		\draw[thick][color=black] (-5 ,2) -- (4,2);
		\draw[thick][color=black] (-5 ,4) -- (4,4);
		
		\draw[thick,dashed][color=blue] (-1.5 ,0) -- (-3,2)--(0,4)--(2,2)--(-1.5,0) ;
		\draw[thick][color=black] (-2.8 ,0) -- (-1.7,1.5) ;
		\draw[thick][color=black] (-2.8 ,0) -- (-1.4,1.5) ;
		\draw[thick][color=black] (-2.8 ,0) -- (-4,1.5) ;
				
		\draw[thick][color=black] (1.8 ,0) -- (.6,1.5) ;
		\draw[thick][color=black] (1.8 ,0) -- (3,1.5) ;
		
			\filldraw (-1.5 ,0) circle (2pt);
			\filldraw (-3,2) circle (2pt);
			\filldraw (-0,4) circle (2pt);
			\filldraw (2,2) circle (2pt);
		
		\draw[thick][color=black] (-.5 ,2) -- (1.2,3.5) ;
		\draw[thick][color=black] (-0.5 ,2) -- (-2,3.5) ;
		\draw[thick] (0,4) node[anchor=south] {\tiny$(x_{k+1}+\theta_{n+2}\De x,t_{n+2})$};
				\draw[thick] (-3,2) node[anchor=south east] {\tiny$(x_{k}+\theta_{n+1}\De x,t_{n+1})$};
						\draw[thick] (2,2) node[anchor=south west] {\tiny$(x_{k+2}+\theta_{n+1}\De x,t_{n+1})$};
								\draw[thick] (-1.1,0) node[anchor=north] {\tiny$(x_{k+1}+\theta_{n}\De x,t_{n})$};

		\end{tikzpicture}
		\caption{This illustrates the diamond (in dotted lines) formed by $(J_2\setminus J_1)\cup (J_1\setminus J_2)$ where $J_i,\,i=1,2$ are I-curves and $J_2$ is immediate successor of $J_1$. }
		 \label{fig-4}
	\end{figure}
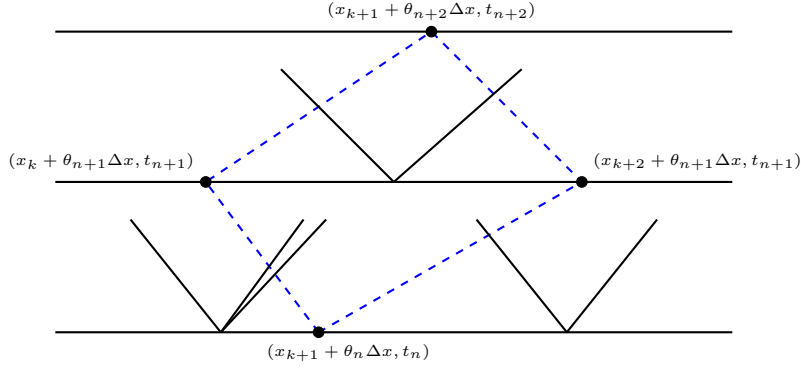
	From \eqref{def:u:n+1} we have
	\begin{equation*}
	u_{n+1,k}=u_{n+\frac{1}{2},k}\mbox{exp}(-\si_{k}^\De\De t)-\frac{q}{m}\frac{1-\mbox{exp}(-\si_{k}^\De\De t)}{\si^\De_{k}}\Psi_{n+1,k}.
	\end{equation*}
	We rewrite $u_{n+1,k}$ as follows
	\begin{equation*}
	u_{n+1,k}=u_{n+\frac{1}{2},k}+\de_{n+1,k}
	\end{equation*}
	 where $\de_{n+1,k}$ is defined as
	 $$\de_{n+1,k}=u_{n+\frac{1}{2},k}\left[\mbox{exp}(-\si_{k}^\De\De t)-1\right]-\frac{q}{m}\frac{1-\mbox{exp}(-\si_{k}^\De\De t)}{\si^\De_{k}}\Psi_{n+1,k}.$$
	Similarly we write $u_{n+1,k+2}=u_{n+\frac{1}{2},k+2}+\de_{n+1,k+2}$ where $\de_{n+1,k+2}$ is defined as
	\begin{equation*}
	\de_{n+1,k+2}=u_{n+\frac{1}{2},k+2}\left[\mbox{exp}(-\si_{k+2}^\De\De t)-1\right]-\frac{q}{m}\frac{1-\mbox{exp}(-\si_{k+2}^\De\De t)}{\si^\De_{k+2}}\Psi_{n+1,k+2}.
	\end{equation*}
We can observe in particular using the definition of the scheme (\ref{6.1}) and (\ref{6.2}) that the strength of the shock waves which cross $J_1\setminus J_2$ are issue of the two Riemann problem with left and right states $[(\varrho_{n+\frac{1}{2},k},u_{n+\frac{1}{2},k}),(\varrho_{n,k+1},u_{n,k+1})] $ and $[(\varrho_{n,k+1},u_{n,k+1}),(\varrho_{n+\frac{1}{2},k+2},u_{n+\frac{1}{2},k+2}) $. We remark also that the strength of the shock waves which cross $J_2\setminus J_1$ are issue of the Riemann problem with left and right state $[(\varrho_{n+1,k},u_{n+1,k}), (\varrho_{n+1,k+2},u_{n+1,k+2})]$.
\\
Now we are going to define $\tilde{J_2}$  which is the same $I$ curves as $J_2$ except that the strength of the shock wave which cross  $\tilde{J_2}\setminus J_1$  is coming from the Riemann problem with left state  $(\varrho_{n+1,k},u_{n+1,k})$ and with right state $(\varrho_{n+1,k+2},u_{n+1,k+2})$ when $\de_{n+1,k}=\de_{n+1,k+2}=0$. 
It is natural now to consider the quantities $F(\tilde{J}_2),V(\tilde{J}_2)$ when $\de_{n+1,k}=\de_{n+1,k+2}=0$. Note that if $\de_{n+1,k},\de_{n+1,k+2}$ does not contribute in solution then the changes in $F,V,Q$ for the diamond $D_{n+1,k}$ is same as the homogeneous case done in \cite{Nisi-Smo}. We recall the following Lemma issue from \cite{Nisi-Smo}.
\begin{lemma}
\label{lemSmo}
If $\e F(J_1)$  is sufficiently small, then:
$$F(\tilde{J_2})\leq F(J_1),$$
with $\tilde{J_2}$ is an immediate successor to $J_1$.
\end{lemma}
\begin{remarka}
	Note that by induction hypothesis $F(\bar{J}_{i}),F(O^n)$ satisfy \eqref{estimate-FJi} and \eqref{estimate-FOn} respectively. Since $J_1$ coincides with $\bar{J}_{i}$ for some $i$, we get $4\e CF(J_1)\leq \min\{C_0,1\}$. Therefore, proof of Lemma \ref{lemSmo} follows from \cite[Lemma 5, page-197]{Nisi-Smo}.
\end{remarka}
 Therefore, we can write 
	\begin{equation}
	\tilde{F}(J_2)=\tilde{V}(J_2)+K\tilde{Q}(J_2)\leq F(J_1)=V(J_1)+KQ(J_1).
	\label{cruSmo}
	\end{equation}
	By using Lemma \ref{lemma:estimate-3} we can now evaluate the difference of the strength between $\tilde{J}_2$ and $J_2$ on $J_2\setminus J_1$
	\begin{align}
	V(J_2)-V(\tilde{J}_2)&\leq 2\abs{\de_{n+1,k+2}-\de_{n+1,k}},\nonumber\\
	Q(J_2)-Q(\tilde{J}_2)&=\sum_{\beta\notin J_2\setminus J_1, \beta\;\mbox{approaching}\;\gamma'_2}|\beta||\gamma'_2|+\sum_{\beta\notin J_2\setminus J_1, \beta\;\mbox{approaching}\;\beta'_2}|\beta||\beta'_2|\\
	&-\sum_{\beta\notin J_2\setminus J_1, \beta\;\mbox{approaching}\;\gamma'_2}|\beta||\tilde{\gamma}_2|-\sum_{\beta\notin J_2\setminus J_1, \beta\;\mbox{approaching}\;\beta'_2}|\beta||\tilde{\beta}_2|\nonumber\\
	&\leq 2\abs{\de_{n+1,k+2}-\de_{n+1,k}} \sum_{\beta\in J_1,\notin J_1\setminus J_2}|\beta|.\nonumber
	\end{align}
	Above $\gamma_2$, $\beta_2$ and $\tilde{\gamma}_2$, $\tilde{\beta}_2$ are respectively the shock waves crossing respectively $J_2$ and $\tilde{J}_2$ on $J_2\setminus J_1$. Hence,
	\begin{equation}
	F(J_2)-F(\tilde{J}_2)\leq 2\abs{\de_{n+1,k+2}-\de_{n+1,k}}+2K \abs{\de_{n+1,k+2}-\de_{n+1,k}} \sum_{\beta\in J_1,\notin J_1\setminus J_2}|\beta|.
	\label{cruSmo1}
	\end{equation}
	Using (\ref{cruSmo}) and (\ref{cruSmo1}) we deduce that:
	\begin{equation}
	F(J_2)\leq F(J_1)+2\abs{\de_{n+1,k+2}-\de_{n+1,k}}+2K\abs{\de_{n+1,k+2}-\de_{n+1,k}} \sum_{\beta\in J_1,\notin J_1\setminus J_2}|\beta|.
	\label{cruSmo2}
	\end{equation}
%
%
%
Let $a_{n,i},b_{n,i}$ be defined as 
		\begin{align}
	a_{n,i}=\left[\mbox{exp}(-\si_{i}^\De\De t)-1\right]u_{n+\frac{1}{2},i} \mbox{ and }b_{n,i}=\frac{q}{m}\frac{1-\mbox{exp}(-\si_{i}^\De\De t)}{\si^\De_{i}}\Psi_{n+1,i}.
	\end{align}
	We observe in particular that $\delta_{n,i}=a_{n,i}+b_{n,i}$.
	From triangular inequality we get
	\begin{equation}
	\abs{a_{n,i+2}-a_{n,i}}\leq \De t\abs{\si_{i+2}-\si_{i}}\abs{u_{n+\frac{1}{2},i+2}}+\De t\abs{\si_{i}}\abs{u_{n+\frac{1}{2},i+2}-u_{n+\frac{1}{2},i}}.
	\label{72}
	\end{equation}
	Similarly, we obtain
	\begin{align}
	\abs{b_{n,i+2}-b_{n,i}}
	&\leq \frac{q}{m}\abs{\frac{1-\mbox{exp}(-\si_{i+2}^\De\De t)}{\si^\De_{i+2}}-\frac{1-\mbox{exp}(-\si_{i}^\De\De t)}{\si^\De_{i}}}\abs{\Psi_{n+1,i+2}}\nonumber\\
	&+ \frac{q}{m}\abs{\frac{1-\mbox{exp}(-\si_{i}^\De\De t)}{\si^\De_{i}}}\abs{\Psi_{n+1,i+2}-\Psi_{n+1,i}}\nonumber\\
	&\leq \frac{q}{m}\abs{\Psi_{n+1,i+2}}\abs{\si^{\De}_{i+2}-\si^{\De}_{i}}\De t^2+\frac{q}{m}\abs{\Psi_{n+1,i+2}-\Psi_{n+1,i}}\De t.
	\label{73}
	\end{align}
	Combining (\ref{cruSmo2}), (\ref{72}) and (\ref{73}) and the fact that $V(J_1)\leq F(J_1)$, we have
		\begin{align}
	&F(J_2)-F(J_1)\nonumber\\
&\leq\biggl( 2\abs{u_{n+\frac{1}{2},k+2}}\abs{\si_{k+2}-\si_{k}}\De t+2\abs{\si_{k}}\abs{u_{n+\frac{1}{2},k+2}-u_{n+\frac{1}{2},k}}\De t\nonumber\\
	&+2\left[\frac{q}{m}\abs{\Psi_{n+1,k+2}}\abs{\si^{\De}_{k+2}-\si^{\De}_{k}}\De t^2+\frac{q}{m}\abs{\Psi_{n+1,k+2}-\Psi_{n+1,k}}\De t\right]\biggl)(1+K \sum_{\beta\in J_1,\notin J_1\setminus J_2}|\beta|).\label{estimate-1}
	\end{align}
%
	We can now estimate $\abs{u_{n+\frac{1}{2},k+2}-u_{n+\frac{1}{2},k}}$ in terms of the strength of the shock wave crossing  $J_1$ by using (\ref{techimp}) (again we assume that $J_1$ is an curve which is a line $k^+\De t$, $k^-\De t$ at infinity)
$$\abs{u_{n+\frac{1}{2},k+2}-u_{n+\frac{1}{2},k}}\leq 4V(J_1)+C'_T\leq 4 F(J_1)+C'_T. $$
Now using Lemma \ref{lemma1} and (\ref{techimp}) we have $\abs{u_{n+\frac{1}{2},k+2}}\leq |u_n^-|+TV(u^\De_{| O_n})\leq C_T+4V(J_1)\leq C_T+4 F(J_1)$. Here $C_T$ is a constant depending only on $T$.
 Then it yields from (\ref{estimate-1}) and choice of $K$:
\begin{align}
&F(J_2)-F(J_1)\nonumber\\
&\leq \bigg( 2C_T\abs{\si_{k+2}-\si_{k}}\De t+8\abs{\si_{k+2}-\si_{k}}\De tF(J_1)+8\abs{\si_{k}}\De tF(J_1)+8\abs{\si_{k}}\De t \,C'_T\nonumber\\
&+2\left[\frac{q}{m}\abs{\Psi_{n+1,k+2}}\abs{\si^{\De}_{k+2}-\si^{\De}_{k}}\De t^2+\frac{q}{m}\abs{\Psi_{n+1,k+2}-\Psi_{n+1,k}}\De t\right]\bigg)(1+4C\e F(J_1)).\label{est-FJj}
\end{align}
Now we set $J_1=\bar{J}_j$ and $J_2=\bar{J}_{j+1}$. Note that by our assumption, $F(\bar{J}_j)$ satisfies \eqref{estimate-FJi} and $F(O^n)$ satisfies \eqref{estimate-FOn}. Due to the choice of $\e$ and the fact $m\De x\leq 2L+2(n+1)\De x\leq 2L+\frac{2T}{\lambda}=L_1$, we have
\begin{align}
1+KF(\bar{J}_j)&\leq 1+4C\e\left[e^{\la A L_1}F(O^n)+B\frac{e^{\la AL_1}-1}{A}\right]\nonumber\\
&\leq 1+4C\e\left[e^{\la A L_1}e^{A_1n\De t}F(O_1)+e^{\la A L_1}B_1\frac{e^{A_1n\De t}-1}{A_1}+B\frac{e^{\la AL_1}-1}{A}\right]\nonumber\\
&\leq 2.\label{est-KFJj}
\end{align}
Hence, we have
\begin{align}
&F(\bar{J}_{j+1})-F(\bar{J}_j)\nonumber\\
&\leq \bigg( 2C_T\abs{\si_{k+2}-\si_{k}}\De t+8\abs{\si_{k+2}-\si_{k}}\De tF(\bar{J}_j)+8\abs{\si_{k}}\De tF(\bar{J}_j)+8\abs{\si_{k}}\De t \,C'_T\nonumber\\
&+2\left[\frac{q}{m}\abs{\Psi_{n+1,k+2}}\abs{\si^{\De}_{k+2}-\si^{\De}_{k}}\De t^2+\frac{q}{m}\abs{\Psi_{n+1,k+2}-\Psi_{n+1,k}}\De t\right]\bigg)(1+KF(\bar{J}_j))\nonumber\\
&\leq 8(C_T+C'_T)\norm{\si^\De }_{\f}\De t+32\norm{\si^\De}_\f\De tF(\bar{J}_j)+\frac{8q}{m}\left[\norm{\Psi^\De}_{\f}\norm{\si^\De}_\f\De t+\norm{\Psi^\De}_\f\right]\De t.
\end{align}
By choice of $A,B$ as in \eqref{def:A-B} we have using the recurrence hypothesis (\ref{estimate-FJi}) at the level $j$
\begin{align*}
F(\bar{J}_{j+1})&\leq (1+A\De t)F(\bar{J}_{j})+B\De t\\
&\leq (1+A\De t)e^{\la Aj\De x}F(O^n)+B\De t\sum\limits_{i=1}^{j}(1+A\De t)^i+B\De t\\
&\leq e^{\la A(j+1)\De x}F(O^n)+B\De t\sum\limits_{i=0}^{j}(1+A\De t)^i.
\end{align*}
This completes proof of Claim \ref{claim-1}.

	\end{proof}
    From \eqref{est-FJj} with $J_1=\bar{J}_{k},J_2=\bar{J}_{k+1}$, we obtain
    \begin{align}
&F(\bar{J}_{k+1})-F(\bar{J}_{k})\nonumber\\
&\leq\biggl( 2\abs{u_{n+\frac{1}{2},k+2}}\abs{\si_{k+2}-\si_{k}}\De t+2\abs{\si_{k}}\abs{u_{n+\frac{1}{2},k+2}-u_{n+\frac{1}{2},k}}\De t\nonumber\\
	&+2\left[\frac{q}{m}\abs{\Psi_{n+1,k+2}}\abs{\si^{\De}_{k+2}-\si^{\De}_{k}}\De t^2+\frac{q}{m}\abs{\Psi_{n+1,k+2}-\Psi_{n+1,k}}\De t\right]\biggl)(1+K \sum_{\beta\in J_1,\notin J_1\setminus J_2}|\beta|).\nonumber\\
\end{align}
	Using  Claim \ref{claim-1}, \eqref{est-KFJj} and the fact that $\abs{u_{n+\frac{1}{2},k+2}}\leq C_T+4 F(J_1)$ we obtain
	\begin{equation}
	\begin{aligned}
	&F(\bar{J}_{k+1})-F(\bar{J}_k)\\
    &\leq 4C_T\abs{\si_{k+2}-\si_{k}}\De t+16\left(e^{\la A L_1} F(O^n)+B\frac{e^{A\la L_1}-1}{A}\right)\abs{\si_{k+2}-\si_{k}}\De t\\
	&+4\norm{\si^\De}_\f\abs{u_{n+\frac{1}{2},k+2}-u_{n+\frac{1}{2},k}}\De t+\frac{4q}{m}\norm{\Psi^\De}_\f\abs{\si^{\De}_{k+2}-\si^{\De}_{k}}\De t^2\\
	&+\frac{4q}{m}\abs{\Psi_{n+1,k+2}-\Psi_{n+1,k+1}}\De t.
	\end{aligned}
	\label{ncru}
	\end{equation}
	Therefore by summing (\ref{ncru}) that we combine with (\ref{techimp}), it gives
	\begin{align}
    &F(O^{n+1})-F(O^n)\nonumber\\
	&\leq 4C_TTV(\si^\De )\De t+4\left(e^{\la A L_1} F(O^n)+B\frac{e^{A\la L_1}-1}{A}\right)TV(\si^\De )\De t\nonumber\\
	&+4\norm{\si^\De}_\f (4 F(O^n)+C'_T)\De t+\frac{4q}{m}\norm{\Psi_{n+1}^\De}_\f TV(\si^\De)\De t^2+\frac{4q}{m}TV(\Psi^\De_{n+1})\De t.\nonumber
	\end{align}
	Let $A_1,B_1>0$ be as in \eqref{def:A1}, \eqref{def:B1}, then we get
	\begin{align}
	A_1&\geq 4e^{\la A  L_1}TV(\si^\De )+16 \norm{\si^\De}_\f,\nonumber\\
	B_1&\geq 4C_TTV(\si^\De )+4B\frac{e^{A\la L_1}-1}{A}TV(\si^\De )+4 C'_T  \norm{\si^\De}_\f+\frac{4q}{m}\norm{\Psi_{n+1}^\De}_\f TV(\si^\De)\De t\nonumber\\
	&+\frac{4q}{m}TV(\Psi^\De_{n+1}).\nonumber
	\end{align}
	Hence,
	\begin{equation*}
	F(O^{n+1})\leq F(O^n)(1+A_1\De t)+B_1\De t.
	\end{equation*}
	This completes the proof of Lemma \ref{lemma:BV-bound}.
\end{proof}
 Once we obtain \eqref{estimate-FOn}, subsequently, we have
\begin{align}
F(O^{n})&\leq e^{A_1n\De t}F(O_1)+B_1\frac{e^{A_1n\De t}-1}{A_1}\nonumber\\
&\leq e^{A_1T}F(O_1)+B_1\frac{e^{A_1T}-1}{A_1}\quad\quad\mbox{ for }1\leq n\leq \frac{T}{\De t}.
\label{supercru}
\end{align}

	
	\begin{remark}\label{remark-bd-psi-u}
		From \eqref{supercru}, we have $\norm{(\varrho^\De,u^\De)}_{L^{\f}(\R\times[0,T])}\leq C(T,\varrho_0,u_0)$. By \eqref{25}, \eqref{26} we get  then we have uniform bound of $\gamma_{n},\xi_n$ depending only on $T,\varrho_0$ and $u_0$. Subsequently, we obtain $\abs{\Psi^\De_{n,i+2}-\Psi^\De_{n,i}}\leq \tilde{C}\De x$ and $\abs{u_{n+1,j}-u_{n+\frac{1}{2},j}}\leq C^*\De t$.
	\end{remark}
	
	For uniform lower bound of approximate density $\varrho^\De$, we have the following lemma.
	\begin{lemma}
		Let $\bar{\e}_0>0$ be satisfying the following
		\begin{equation*}
		\sup\limits_{x\in\R}s(\varrho_0(x),u_0(x))-\inf\limits_{x\in\R}r(\varrho_0(x),u_0(x))+2C^*T<\frac{\sqrt{\gamma}}{\bar{\e}_0}
		\end{equation*}
		where $C^*$ is as in Remark \ref{remark-bd-psi-u}.	Then there exists $\underline{\underline{\varrho}}>0$ such that $\varrho^\De(x,t)\geq \underline{\underline{\varrho}}$ for a.e. $(x,t)\in\R\times[0,T]$.
	\end{lemma}
	\begin{proof}
		Since in the first step of approximation, we solve Riemann data for homogeneous problem \eqref{eqn-Euler-1}--\eqref{eqn-Euler-2}, by Lemma \ref{lemma:Euler-1} 
		\begin{align}
		\inf\limits_{j\in\mathbb{Z}}r(\varrho_{n,j},u_{n,j})&\leq r(\varrho_{n+\frac{1}{2},j},u_{n+\frac{1}{2},j})\nonumber\\
		&\leq s(\varrho_{n+\frac{1}{2},j},u_{n+\frac{1}{2},j})\leq \sup\limits_{j\in\mathbb{Z}}s(\varrho_{n,j}(x),u_{n,j}(x)),\nonumber
		\end{align}
		for all $j\in\mathbb{Z}$. In the second step $\varrho$-variable remains same and $u$-variable is changed by $\de_{n,j}=u_{n+\frac{1}{2},j}-u_{n,j}$. By Remark \ref{remark-bd-psi-u}, $\abs{\de_{n,j}}\leq C^*\De t$. Therefore,
		\begin{align*}
		\inf\limits_{x\in\R}r(\varrho_{n,j}(x),u_{n,j}(x))-C^*\De t&\leq r(\varrho_{n+1,j}(x),u_{n+1,j}(x))\\
		&\leq s(\varrho_{n+1,j}(x),u_{n+1,j})\leq \sup\limits_{x\in\R}s(\varrho_{n,j}(x),u_{n,j}(x))+C^*\De t.
		\end{align*}
		Hence, we have
		\begin{align*}
		\inf\limits_{x\in\R}r(\varrho_0(x),u_0(x))-C^*n\De t&\leq r(\varrho_{n,j}(x),u_{n,j}(x))\\
		&\leq s(\varrho_{n,j}(x),u_{n,j})\leq \sup\limits_{x\in\R}s(\varrho_0(x),u_0(x))+C^*n\De t,
		\end{align*}
		equivalently,
		\begin{align}
		&s(\varrho_{n,j}(x),u_{n,j})-r(\varrho_{n,j+2}(x),u_{n,j+2})\nonumber\\
		&\geq\inf\limits_{x\in\R}r(\varrho_0(x),u_0(x))-\sup\limits_{x\in\R}s(\varrho_0(x),u_0(x))-2C^*T\nonumber\\
		&>-\frac{\sqrt{\gamma}}{\bar{\e}_0}.\label{lower-bound-cal1}
		\end{align}
		Since the bound in \eqref{lower-bound-cal1} does not depend on $n$ and $\De t$, we get a uniform lower bound $\underline{\underline{\varrho}}>0$ such that $\varrho_{n+\frac{1}{2},j+1}\geq \underline{\underline{\varrho}}$ for all $n\geq0, j\in\mathbb{Z}$.
	\end{proof}

\subsection{Proof of Theorem \ref{theorem-1}}
Note that from (\ref{supercru}) and Remark \ref{remimp1a}
we have a uniform $BV$ bound in $n$ of the sequence of functions $(\varrho^{\De,\theta}(t_n,\cdot),u^{\De,\theta}(t_n,\cdot))$. 
We deduce solving the different Riemann problems that for any $t\in[0,T]$ there exists $C_{1,T}$ such that for any sequence $\theta$:
\begin{equation}
\|(\varrho^{\De,\theta}(t,\cdot),u^{\De,\theta}(t,\cdot))\|_{TV(\R)}\leq C_{1,T}\;\;\mbox{for}\;0<t\leq T.
\end{equation}
Similarly from (\ref{L-infty-varrho}), (\ref{L-infty-u}), (\ref{supercru}) and the resolution of the Riemann problem there exists $C_{2,T}>0$ such that for any $t\in[0,T]$ we have:
\begin{equation}
\|(\varrho^{\De,\theta}(t,\cdot),u^{\De,\theta}(t,\cdot))\|_{L^\infty(\R)}\leq C_{2,T}.
\end{equation}
This last estimate implies that there exists $C_{3,T}>0$ such that:
\begin{equation}
\max( \|\lambda_1(\rho^{\De,\theta})\|_{L^\infty([0,T]\times\R)},\| \lambda_2(\rho^{\De,\theta})\|_{L^\infty([0,T]\times\R)})\leq C_{3,T}
\label{CFLa}
\end{equation}
We can now fix the CFL condition on $\lambda$, indeed we wish to solve each Riemann problem at any time $t_n=n\De t\leq T$ such that there is no interaction between the different Riemann problem. It suffices to choose $\lambda$ such that:
\begin{equation}
\lambda<\frac{1}{C_{3,T}}.
\label{CFL}
\end{equation}
Now the rest of the proof is similar to the one given in \cite{P-R-V}, in particular $(\varrho^{\De,\theta},u^{\De,\theta},\Psi^{\De,\theta})$ converges weakly when $\De x$ goes to $0$ to a weak solution $(\varrho,u,\Psi)$ of the Euler Poisson system provided that the sequence $\theta$ is suitably chosen (indeed the convergence is true for almost every sequence $\theta$ for the uniform probability measured $d\nu$, product
of the uniform measures $d m_j = \frac{1}{2} da_j$ on each factor $(-1, 1)$).

\section{Initial boundary value problem}\label{sec:IBVP}
In this section, we discuss the initial boundary value problem for the Euler-Poisson system \eqref{eqn-EP-1}--\eqref{eqn-EP-3}. More precisely, we consider the system \eqref{eqn-EP-1}--\eqref{eqn-EP-3} in the domain $(0,\f)\times(0,\f)$ and the following additional initial and boundary condition.  
\begin{equation}\label{eq:initial-boundary-data}
	\left\{\begin{array}{rll}
		(\varrho,u)(0,x)&=(\varrho_0,u_0)(x),&x\in(0,\f),\\
		\varrho u(t,0)&=m_b(t), &t\in(0,\f).
	\end{array}\right.
\end{equation}
The corresponding weak formulation is as below:
\begin{equation}\label{def:wk-soln-IBVP-1}
	\int\limits_{0}^{\f}\int\limits_{0}^{\f}\left[\varrho\pa_t\varphi+\varrho u\pa_x\varphi\right]\,dxdt+\int\limits_{0}^{\f}\varrho_0(\cdot)\varphi(0,\cdot)\,dx+\int\limits_{0}^{\f}m_b(t)\varphi(t,0)\,dt=0,
\end{equation}
for $\varphi\in C^1_c([0,\f)\times[0,\f))$ and 
\begin{equation}\label{def:wk-soln-IBVP-2}
	\int\limits_{0}^{\f}\int\limits_{0}^{\f}\left[\varrho u\pa_t\psi+(\varrho u^2+P(\varrho))\pa_x\psi+(\si \varrho u+\frac{q}{m}\varrho\Psi)\psi \right]\,dxdt+\int\limits_{0}^{\f}\varrho_0u_0(\cdot)\psi(0,\cdot)\,dx=0,
\end{equation}
for $\psi\in C^1_c([0,\f)\times[0,\f))$ with $\psi(t,0)=0$ for all $t\geq0$ along with $$\Psi(t,x)=\Psi^-(t)-\frac{q}{e}\int\limits_{0}^{x}(\varrho-\mu)\,dy.$$ 
Similar to the Cauchy problem, we define the far field condition as follows.
%
%
\begin{align}
	&\mu(x)=
		\mu^+\mbox{ for }x>L,\quad
	\si(x)=\si^+\mbox{ for }x>L,\label{IBVP-eq:bdy-cond1}\\
	&\varrho_0(x)=\mu^+\mbox{ for }x>L,
	\quad
	u_0(x)=u^+\mbox{ for }x>L,\label{IBVP-eq:bdy-cond2}
\end{align}
for some $L>0$. We prove that these properties are preserved for any positive time $t\geq0$:
\begin{equation}\label{IBVP-3}
		\varrho(t,x)=\mu^+,\;t>0,\;x<L(t),\quad u(t,x)=u^+(t),t>0,x>L(t),
\end{equation}
with $L(t)$ depending on the time $t$.
Therefore (\ref{eqn-EP-3}) yields:
\begin{equation}
	\Psi(t,x)=\Psi^-(t)-\int^x_{0}\frac{q}{e}(\varrho(t,y)-\mu(y))dy\mbox{ for } t>0,x\in(0,\f).
\end{equation}
The electric field at $x=0$, $\Psi^-(t)$ is assumed to be known. Similar to the Cauchy problem as discussed in Section \ref{sec:intro} we can compute the values of $\varrho$, $u$, and $\Psi$ at $x=+\f$.
\begin{equation}
	\begin{cases}
		\begin{aligned}
			&
			u^{+}(t)=u_0^{+}-\frac{q}{e}\int^t_0e^{-\sigma^+(t-s)}\Psi^+(s) ds,\\
			&\Psi^+(t)=\Psi^+(0)+\int^t_0\frac{q}{e}(\mu^+u^+(s)-m_1(s))ds+\Psi^-(t)-\Psi^-(0),\\
			&\Psi^+(0)=\Psi^-(0)-\int^{+\infty}_{0}\frac{q}{e}(\varrho_0(y)-\mu(y))dy.
		\end{aligned}
	\end{cases}
\end{equation}
Now we are ready to state our result on existence of BV solution to initial boundary value problem for Euler-Poisson system in 1-D.
\begin{theorem}\label{theorem-2}
	Let $T>0$ and $(\varrho_0,u_0)$ such that  $(\varrho_0,u_0)\in BV((0,\f),\R_+\times\R)$ and there exists $M,\bar{\varrho},\underline{\varrho}\in(0,\f)$ such that
	$$
	\abs{u_0(x)}\leq M,\,0<\underline{\varrho}\leq \varrho_0(x)\leq \bar{\varrho}\mbox{ for a.e. }x\in(0,\f).
	$$
	Furthermore $\sigma,\mu$, $m_b$ and $\Psi^-$ verify the following assumptions:
	\begin{itemize}
		\item $\sigma\geq 0$, $\sigma\in BV(0,\f)$, $\mu \geq 0$, $\mu  \in BV(0,\f)$,
		\item $m_b\geq c_0>0$ and $m_b\in BV(0,\f)$,
		\item $\Psi^-\in BV(0,\infty)$.
	\end{itemize} Then there exists $\gamma_0\in(1,2)$ such that the following holds: for any $\gamma\in(1,\gamma_0]$
	there exists a BV weak solution (in the sense of \eqref{def:wk-soln-IBVP-1}--\eqref{def:wk-soln-IBVP-2}) $(\varrho,u,\Psi)$ of \eqref{eqn-EP-1}--\eqref{eqn-EP-3} and \eqref{eq:initial-boundary-data} satisfying \eqref{IBVP-eq:bdy-cond1}--\eqref{IBVP-eq:bdy-cond2} on $[0,T]$.
\end{theorem}
The initial-boundary value problem for the isentropic Euler equation in Lagrangian variables has been studied in \cite{Nishida,NS-bdy} with the boundary condition as in \eqref{eq:initial-boundary-data}. In \cite{Nishida,NS-bdy}, the existence of BV solutions has been shown for large initial data. For general $n\times n$ hyperbolic system of conservation laws, the initial-boundary problem has been studied \cite{Amadori} for initial and boundary data with small total variation. Stability of solutions with respect to initial and boundary data has been studied in \cite{AC,DM-1}. 

Proof of Theorem \ref{theorem-2} follows similarly to that of Theorem \ref{theorem-1} with an appropriate adaptation of Glimm scheme in quarter plane. Here, we point out the changes and omit the details as the rest follows by exactly same arguments. 

Let $x_i,t_n$ be the space and time grid points defined as in section \ref{sec:scheme}. We recall the definition of intervals $I_{n,i}$ as $I_{n,i}:=(x_{i-1},x_{i+1})$ for $i$ is integer $i\geq1$ such that $n+i$ is even. We additionally define, $I_{n,0}$ is defined as $I_{n,0}:=(0,x_{1})$ when $n$ is even. Note that 
\begin{equation*}
	(0,\f)=\left\{\begin{array}{ll}
		\bigcup\limits_{i=2j+1,j\in\mathbb{Z},j\geq0}I_{n,i}&\mbox{ when $n$ is odd},\\
		\left(\bigcup\limits_{i=2j,j\in\mathbb{Z},j\geq1}I_{n,i}\right)\cup I_{n,0}&\mbox{ when $n$ is even}.
	\end{array}\right.
\end{equation*} Consider $\varrho^\De_0,u^\De_0,\sigma^\De,\mu^\De$ as in \eqref{def:approx-initial-1}--\eqref{def:approx-initial-2}. We consider $\Psi^\pm_n$ be defined as in \eqref{12}. The discretization of the velocity at the positive infinity is given by:
\begin{align}
	u^{+}_{n+1}=u^+_{n}\mbox{exp}(-\si^+\De t)-\frac{q}{e}\frac{1-\mbox{exp}(-\si^+\De t)}{\si^+}\Psi^+_{n+1}\mbox{ for }n\geq0.
\end{align}
We can approximate the boundary data $m_b$ as below.
\begin{equation}
	m_b^\De(t):=\sum\limits_{n\geq0}m_b^n\chi_{(n\De t,(n+1)\De t)}(t)\mbox{ where }m_b^n=m_b\left(\frac{(2n+1)\De t}{2}\right)\mbox{ for }n\geq0.
\end{equation}
Next, our goal is to define $(\varrho^\De_{n+1},u^\De_{n+1})$. Recall the notation $\mathcal{R}[U_l,U_r](t,x)$ representing the solution to the Riemann problem \eqref{eqn-Euler-1}--\eqref{eqn-Euler-2} and \eqref{def:Rie-data}. To take into account the boundary let us first consider the following data,
\begin{equation}\label{boundary-Rie}
	(\varrho(0,x),u(0,x))=U_+:=(\varrho_+,u_+),\mbox{ for }x>0\mbox{ and }\varrho(t,0)u(t,0)=\overline{m}_b\mbox{ for }t>0.
\end{equation}
For $t>0,x>0$, we define $\mathcal{R}_b[\overline{m}_b,U_+](t,x)$ as solution of the initial boundary value problem \eqref{eqn-Euler-1}--\eqref{eqn-Euler-2}, \eqref{boundary-Rie} which consists of only 2-wave. We have already made remark on the existence of $\mathcal{R}[U_l,U_R]$ in section \ref{sec:scheme}. Similarly, for any $\varrho_+,u_+$ and $\overline{m}_b\geq c_0$, there exists $\gamma_0>1$ such that $\mathcal{R}_b[\overline{m}_b,U_+]$ exists whenever $\gamma\in(1,\gamma_0)$ (indeed, we avoid the appearance of vacuum in the solution). Note that for $\overline{m}_b\geq c_0>0$ ensure the unique solvability when we allow only 2-wave. We can show this in the following way. First consider the case when $\overline{m}_b>\varrho_+u_+$. If $(\varrho_+,u_+)$ lies on the $S_2$ curve starting from $(\varrho,u)$, then we have  $u-u_+=\varrho_+^\e\sqrt{\frac{(\al-1)(\al^{\ga}-1)}{\al}}=:\varrho_+^\e G(\al)$ with $\al=\frac{\varrho}{\varrho_+}>1$. Observe that $G$ is a strictly increasing continuous function from $[1,\f)$ to $[0,\f)$. Hence, there exists $G^{-1}$ which is strictly increasing and continuous function from $[0,\f)$ to $[1,\f)$. We write $\varrho u=\varrho_+G^{-1}\left(\frac{u-u_+}{\varrho_+^\e}\right)u$. Since for $u>0$, the function $Y(u)=\varrho_+ G^{-1}\left(\frac{u-u_+}{\varrho_+^\e}\right)u$ is strictly increasing and continuous with $\lim\limits_{u\rr\f} Y(u)=\f$ we conclude that $Y$ is a bijective map from $(\max\{0,\varrho_+u_+\},\f)$ to $(\max\{0,\varrho_+u_+\},\f)$. Hence, there exists $u_*>0$ such that $Y(u_*)=\overline{m}_b$ and we recover $\varrho_*$ as $\varrho_*=\overline{m}_b/u_*\in(0,\f)$. The other follows in simpler way as we need to consider the rarefaction curve with $\gamma$ sufficiently close to $1$.  
\begin{remark}\label{remark-counter}
	We remark that even for $m_b\geq c_0>0$, there can be infinitely many solutions to the Riemann type problem \eqref{eqn-Euler-1}--\eqref{eqn-Euler-2}, \eqref{boundary-Rie} if we allow 1-wave to appear in the solution. To show this, we consider $\varrho_+=1,u_+=2$ and $\bar{m}_b=4$ along with $\gamma=1+2\e<2$ or equivalently, $\e\in(0,1/2)$. Then we choose $\varrho_b=\left(\frac{4}{D\sqrt{\gamma}}\right)^\frac{1}{1+\e}$ where $D>1$ will be chosen later. Now, let $D_1\in\left(1,\frac{2^\frac{1-\e}{1+\e}}{(\sqrt{\gamma})^{\frac{1}{1+\e}}}\right)$ and $D$ be a solution to 
\begin{equation}\label{eq-D}
	H(D):=D^{\frac{1}{1+\e}}-\frac{1}{\e}D^{-\frac{\e}{1+\e}}=\left[-\frac{1}{\e}+\frac{2}{D_1\sqrt{\gamma}}\right]\left(\frac{4}{\sqrt{\gamma}}\right)^{-\frac{\e}{1+\e}}.
\end{equation}
We note that $\frac{2}{D_1\sqrt{\gamma}}\leq 2\leq \frac{1}{\e}$ and
\begin{equation*}
	0\geq \left[-\frac{1}{\e}+\frac{2}{D_1\sqrt{\gamma}}\right]\left(\frac{4}{\sqrt{\gamma}}\right)^{-\frac{\e}{1+\e}}=-\frac{1}{\e}\left(\frac{4}{\sqrt{\gamma}}\right)^{-\frac{\e}{1+\e}}+\frac{2^\frac{1-\e}{1+\e}}{D_1(\sqrt{\gamma})^{\frac{1}{1+\e}}}\geq -\frac{1}{\e}+1.
\end{equation*} We also observe that $H(1)=-\frac{1}{\e}+1$ and $H(D)\rr\f$ as $D\rr\f$. Hence, there is a solution $D>1$ satisfying \eqref{eq-D}. From the choice of $D$ and $D_1$, we have
\begin{align*}
	r_b=\frac{\bar{m}_b}{\varrho_b}-\frac{\sqrt{\gamma}}{\e}\left(\varrho_b^\e-1\right)&=\frac{4}{\varrho_b}-\frac{\sqrt{\gamma}}{\e}\left(\varrho_b^\e-1\right)=\sqrt{\gamma}D\varrho_b^{\e}-\frac{\sqrt{\gamma}}{\e}\left(\varrho_b^\e-1\right)\\
	&=\sqrt{\ga}\left[D \left(\frac{4}{D\sqrt{\gamma}}\right)^\frac{\e}{1+\e}-\frac{1}{\e} \left(\frac{4}{D\sqrt{\gamma}}\right)^\frac{\e}{1+\e}+\frac{1}{\e}\right]\\
	&=\frac{2}{D_1}<2=r_+=u_+-\frac{\sqrt{\gamma}}{\e}(\varrho_+^\e-1),\\
	\la_1\left(\varrho_b,\frac{\bar{m}_b}{\varrho_b}\right)=\frac{\bar{m}_b}{\varrho_b}-\sqrt{\gamma}\varrho_b^\e&=\frac{4}{\varrho_b}-\sqrt{\gamma}\varrho_b^\e=\sqrt{\gamma}D\varrho_b^{\e}-\sqrt{\gamma}\varrho_b^\e\geq0.
\end{align*}
Since $r_b<r_+$, the Riemann problem $\mathcal{R}[U_b,U_+]$ with $U_b=(\varrho_b,\bar{m}_b/\varrho_b),U_+=(\varrho_+,u_+)$ consists of 1-rarefaction and a 2-wave. Since $\la_1(U_b)\geq0$, the left most characteristic of the 1-rarefaction has positive slope and hence, $\mathcal{R}[U_b,U_+](t,0)=U_b$. This concludes our remark on infinitely many choices of solution if we allow 1-wave. Note that for each choice of $D_1\in \left(1,\frac{2^\frac{1-\e}{1+\e}}{(\sqrt{\gamma})^{\frac{1}{1+\e}}}\right)$, we get a $D>1$ and hence a $\varrho_b$. Therefore, there are infinitely many choices of $\varrho_b$ and consequently, there are infinitely many solutions.
\end{remark}
Now we are ready to define the solutions at each time step. For this purpose, let us consider a randomly chosen sequence $\{\theta_n\}\subset[-1,1]$. Let $\Psi_{n+1,i}$ be defined in a similar manner as in section \ref{sec:scheme}.

\noindent\underline{When $n$ is even:} for all odd integer $i\geq 3$ we define,
\begin{equation}
	(\varrho_{n+\frac{1}{2},i},u_{n+\frac{1}{2},i})=\mathcal{R}[U_{n,i-1},U_{n,i+1}](\De t,\theta_{n}\De x).
\end{equation}
For $i=1$, we set $(\varrho_{n+\frac{1}{2},1},u_{n+\frac{1}{2},1})=\mathcal{R}[U_b,U_{n,2}](\De t,\theta_{n}\De x)$ where $U_b=\mathcal{R}_b[m_b^n,U_{n,2}](\De t,0)$. Now for $(n+1)$-th step, we can define as in section \ref{sec:scheme}, for all odd integer $i\geq 1$
\begin{equation}
	\varrho_{n+1,i}=\varrho_{n+\frac{1}{2},i}\mbox{ and }u_{n+1,i}=u_{n+\frac{1}{2},i}\mbox{exp}(-\si_{i}^\De\De t)-\frac{q}{e}\frac{1-\mbox{exp}(-\si_{i}^\De\De t)}{\si^\De_{i}}\Psi_{n+1,i}.
\end{equation}

\noindent\underline{When $n$ is odd:} for all even integer $i\geq 2$ we define,
\begin{equation}
	(\varrho_{n+\frac{1}{2},i},u_{n+\frac{1}{2},i})=\mathcal{R}[U_{n,i-1},U_{n,i+1}](\De t,\theta_{n}\De x).
\end{equation}
For $i=0$, we set $(\varrho_{n+\frac{1}{2},0},u_{n+\frac{1}{2},0})=\mathcal{R}_b[m_b^n,U_{n,2}](\De t,0)$. Now for $(n+1)$-th step, we can define as in section \ref{sec:scheme}, for all odd integer $i\geq 0$
\begin{equation}
	\varrho_{n+1,i}=\varrho_{n+\frac{1}{2},i}\mbox{ and }u_{n+1,i}=u_{n+\frac{1}{2},i}\mbox{exp}(-\si_{i}^\De\De t)-\frac{q}{e}\frac{1-\mbox{exp}(-\si_{i}^\De\De t)}{\si^\De_{i}}\Psi_{n+1,i}.
\end{equation}
%

In next two lemmas we show estimate for boundary interaction.
\begin{lemma}\label{lemma-ibvp-1}
	Let $(r_0,s_0),(r_1,s_1)\in\R^2$ be two points in $r$-$s$ plane  correspond to $(\varrho_0,u_0)$ and $(\varrho_1,u_1)$ respectively. Suppose that the Riemann problem $U_L=(\varrho_0,u_0),U_R=(\varrho_1,u_1)$ has solution with 1-wave strength $\si_1$ and 2-wave strength $\si_2$. Suppose the Riemann problem $U_L^\p:=(\varrho,u)$, $U_R=(\varrho_1,u_1)$ consists only 2-shock with strength $\si$. We assume $\varrho u=\varrho_0 u_0\geq c_0>0$. Then we have
	\begin{equation}
		\abs{\si}\leq (\si_1)_-+(\si_2)_-+\e C_0[(\si_1)_-]^2\mbox{ for some }C_0>0.
	\end{equation} 
\end{lemma}
 The Lemma \ref{lemma-ibvp-1} has been proved in \cite{NS-bdy} via a geometrical approach. Here, we give an analytical version of it.
\begin{proof}
	We first consider the case when $\si_1<0$ and $\si_2=0$. Let us define $r=r(\varrho,u)$ and $s=s(\varrho,u)$. Recall that
    \begin{align}
    	s_0-s_1&=\varrho_0^\e\left[\sqrt{\frac{(\al_1-1)(\al_1^{\ga}-1)}{\al_1}}-\sqrt{\ga}\frac{\al_1^{\e}-1}{\e}\right],\mbox{ where }\al_1=\frac{\varrho_1}{\varrho_0},\\
    	r-r_1&=\varrho^\e\left[\sqrt{\frac{(1-\al)(1-\al^{\ga})}{\al}}-\sqrt{\ga}\frac{1-\al^{\e}}{\e}\right],\mbox{ where }\al=\frac{\varrho_1}{\varrho}.
    \end{align}
By a change of variable we can see that
\begin{equation}\label{def:f}
	    	r-r_1=\varrho_1^\e\left[\sqrt{\frac{(\al_2-1)(\al_2^{\ga}-1)}{\al_2}}-\sqrt{\ga}\frac{\al_2^{\e}-1}{\e}\right]=:\varrho_1^\e f(\al_2)\mbox{ where }\al_2=\frac{\varrho}{\varrho_1}.
\end{equation}
We also see that $s_0-s_1=\varrho_1^\e \frac{f(\al_1)}{\al_1^\e}$. Now, we define $\bar{f}$ as
\begin{equation}\label{def:f-bar}
	\bar{f}(\al)=\left[\sqrt{\frac{(\al-1)(\al^{\ga}-1)}{\al}}+\sqrt{\ga}\frac{\al^{\e}-1}{\e}\right]\mbox{ for }\al>1.
\end{equation}
We can check that $f^\p,\bar{f}^\p \geq0$ (see \cite{Nisi-Smo}). Since, $\al_1,\al_2>1$ we have $\varrho>\varrho_1>\varrho_0$. Hence, we have
\begin{equation*}
	r+s-(r_0+s_0)=2u-2u_0=\frac{2\varrho u}{\varrho}-\frac{2\varrho_0u_0}{\varrho_0}=2\varrho_0u_0\left[\frac{1}{\varrho}-\frac{1}{\varrho_0}\right]\leq0.
\end{equation*}
Subsequently,
\begin{align*}
	\abs{\si}-\abs{\si_1}=(s-s_1)-(r_0-r_1)&=s-s_0+s_0-s_1-(r_0-r+r-r_1)\\
	&=(r+s)-(r_0+s_0)+(s_0-s_1)-(r-r_1)\\
	&\leq (s_0-s_1)-(r-r_1).
\end{align*}
Therefore, we get
\begin{equation}\label{est-si-si1}
	\abs{\si}-\abs{\si_1}\leq (s_0-s_1)-(r-r_1)=\varrho_1^\e\left[\frac{f(\al_1)}{\al_1^\e}-f(\al_2)\right].
\end{equation}
It can be seen that for $\al_2>\al_1$ we have $	\abs{\si}-\abs{\si_1}\leq 0$. Now, we consider the case when $\al_2\leq \al_1$. We note that $\varrho_0^\e f(\al_1)=g_1(\abs{\si_1},\varrho_1)$, $\varrho_1^\e f(\al_2)=g_1(\abs{\si},\varrho_0)$ with $\abs{\si_1}=\varrho_0^\e\bar{f}(\al_1)$ and $\abs{\si}=\varrho_1^\e\bar{f}(\al_2)$. Without loss of generality we assume that $\abs{\si_1}\leq\abs{\si}$. Then, from the definition of $\bar{f}$ as in \eqref{def:f-bar}, we get the following inequality as $\varrho_1=\al_1\varrho_0$,
\begin{equation*}
	\frac{1}{\al^\e_1}\left[\sqrt{\frac{(\al_1-1)(\al_1^{\ga}-1)}{\al_1}}+\sqrt{\ga}\frac{\al_1^{\e}-1}{\e}\right]\leq \left[\sqrt{\frac{(\al_2-1)(\al_2^{\ga}-1)}{\al_2}}+\sqrt{\ga}\frac{\al_2^{\e}-1}{\e}\right].
\end{equation*}
Hence, from \eqref{est-si-si1}, we have
\begin{align*}
	&\frac{1}{\varrho_1^\e}(\abs{\si}-\abs{\si_1})\leq \frac{f(\al_1)}{\al_1^\e}-f(\al_2)\\
	&=\frac{1}{\al^\e_1}\left[\sqrt{\frac{(\al_1-1)(\al_1^{\ga}-1)}{\al_1}}-\sqrt{\ga}\frac{\al_1^{\e}-1}{\e}\right]- \left[\sqrt{\frac{(\al_2-1)(\al_2^{\ga}-1)}{\al_2}}-\sqrt{\ga}\frac{\al_2^{\e}-1}{\e}\right]\\
	&=	\frac{1}{\al^\e_1}\left[\sqrt{\frac{(\al_1-1)(\al_1^{\ga}-1)}{\al_1}}+\sqrt{\ga}\frac{\al_1^{\e}-1}{\e}\right]- \left[\sqrt{\frac{(\al_2-1)(\al_2^{\ga}-1)}{\al_2}}+\sqrt{\ga}\frac{\al_2^{\e}-1}{\e}\right]\\
	&-\frac{2\sqrt{\gamma}}{\e}\left[\frac{\al_1^\e-1}{\al_1^\e}-(\al_2^\e-1)\right]\\
	&\leq \frac{2\sqrt{\gamma}}{\e}\left[\al_2^\e+\frac{1}{\al_1^\e}-2\right]\leq \frac{2\sqrt{\gamma}}{\e}(\al_1^\e-1)^2\mbox{ since }\al_2\leq \al_1\mbox{ and }\al_1>1. 
\end{align*} 
We note that $\abs{\al_1-1}\leq C_0\abs{\si_1}$ and $\abs{\al_1^\e-1}^2\leq \e^2\abs{\al_1-1}^2$ for all $\al_1\geq1$. Therefore, we have
\begin{equation*}
	\abs{\si}-\abs{\si_1}\leq \e C_0\abs{\si_1}^2.
\end{equation*}
Rest of the cases follows in a similar way as in Lemma \ref{lemma:estimate-1}--\ref{lemma:estimate-3}.
\end{proof}
By a similar argument as in Lemma \ref{lemma:estimate-1}, we can obtain the following estimate.
\begin{lemma}\label{lemma-ibvp-2}
	Let $(r_0,s_0),(r_1,s_1),(r_2,s_2)\in\R^2$ be two points in $r$-$s$ plane  correspond to $(\varrho_0,u_0)$, $(\varrho_1,u_1)$ and $(\varrho_2,u_2)$ respectively. Suppose that the Riemann problem $U_L=(\varrho_0,u_0)$, $U_R=(\varrho_1,u_1)$ consists of 2-wave strength $\si_2$ and the Riemann problem $\tilde{U}_L:=(\varrho_2,u_2)$, $U_R=(\varrho_1,u_1)$ consists only 2-shock with strength $\si$. We further assume that $\varrho_0u_0,\varrho_2u_2\geq c_0>0$ and $\underline{\varrho}\leq \varrho_0,\varrho_1,\varrho_2\leq \overline{\varrho}$. Then we have
	\begin{equation}\label{estimate-si-1}
		\abs{\si}\leq C_1\abs{\varrho_2u_2-\varrho_0u_0}+(\si_2)_-,
	\end{equation} 
   for some $C_1>0$.
\end{lemma}
\begin{proof}
	We only give a proof for the case when $\si_2<0$. We first note that if  $s_1\leq s_2\leq s_0$ then we have $\abs{\si}=s_2-s_1\leq s_0-s_1=\abs{\si_2}$. Hence, the estimate \eqref{estimate-si-1} follows. Now, we focus on the case when $s_2\geq s_0\geq s_1$. From \eqref{shock-curve-u} and \eqref{shock-curve-2} we have
	\begin{align}
		&u_0-u_1=\varrho_1^\e\sqrt{\frac{(\al_1^{\gamma}-1)(\al_1-1)}{\al_1}},\,u_2-u_1=\varrho_1^\e\sqrt{\frac{(\al_2^{\gamma}-1)(\al_2-1)}{\al_2}},\\
		&s_0-s_1=\varrho_1^\e\left[\sqrt{\frac{(\al_1^{\gamma}-1)(\al_1-1)}{\al_1}}+\frac{\sqrt{\ga}}{\e}(\al_1^\e-1)\right],\\
		&s_2-s_1=\varrho_1^\e\left[\sqrt{\frac{(\al_2^{\gamma}-1)(\al_2-1)}{\al_2}}+\frac{\sqrt{\ga}}{\e}(\al_2^\e-1)\right],
	\end{align}
where $\al_1=\frac{\varrho_0}{\varrho_1}>1$ and $\al_2=\frac{\varrho_2}{\varrho_1}>1$. Note that $s_2-s_1\geq s_0-s_1$. Since the map $\al\mapsto \sqrt{\frac{(\al^{\gamma}-1)(\al-1)}{\al}}+\frac{\sqrt{\ga}}{\e}(\al^\e-1)$ is increasing we have $\al_2\geq \al_1$.  Again by the increasing property of the map $\al\mapsto \sqrt{\frac{(\al^{\gamma}-1)(\al-1)}{\al}}$ we get $u_2-u_1\geq u_0-u_1$, or equivalently, $u_2\geq u_0$. From \eqref{shock-curve-2} we have
\begin{align*}
	&r_0-r_1=\varrho_1^\e\left[\sqrt{\frac{(\al_1^{\gamma}-1)(\al_1-1)}{\al_1}}-\frac{\sqrt{\ga}}{\e}(\al_1^\e-1)\right],\\
	&r_2-r_1=\varrho_1^\e\left[\sqrt{\frac{(\al_2^{\gamma}-1)(\al_2-1)}{\al_2}}-\frac{\sqrt{\ga}}{\e}(\al_2^\e-1)\right].
\end{align*}
Since the function $\al\mapsto\sqrt{\frac{(\al^{\gamma}-1)(\al-1)}{\al}}-\frac{\sqrt{\ga}}{\e}(\al^\e-1)$ we have $r_2-r_1\geq r_0-r_1$ or equivalently, $r_2\geq r_0\geq r_1$.  We consider $r_3=r(\varrho_0,u_2)$ and $s_3=s(\varrho_0,u_2)$, then observe that $r_2+s_2=r_3+s_3=2u_2$ and $s_0-r_0=s_3-r_3=\frac{2\sqrt{\gamma}}{\e}(\varrho_0^{\e}-1)$. Suppose $r+s=2u_2$ line intersects with the line $r=r_0$ at the point $(r_0,\bar{s})$. Then we have $\bar{s}-s_2=r_2-r_0\geq 0$.  Since $s_0-r_0=s_3-r_3$, we get $s_3-s_0=r_3-r_0$. We also note that $r_3+s_3=r_0+\bar{s}$. Subsequently, $\bar{s}-s_3=r_3-r_0$. Hence, we have 
	\begin{equation}\label{eqn-s2s0}
		s_2-s_0\leq \bar{s}-s_0=\bar{s}-s_3+s_3-s_0=2(r_3-r_0)=2(u_2-u_0).
	\end{equation} 
We have already seen that $\al_2\geq \al_1>1$, hence, we get $\varrho_2\geq \varrho_0>\varrho_1$. Since $u_0\geq 0$, we obtain
\begin{align*}
	u_2-u_0=\frac{1}{\varrho_2}\left[\varrho_2 u_2-\varrho_2 u_0\right]&=\frac{1}{\varrho_2}\left[\varrho_2 u_2-\varrho_0 u_0+(\varrho_0-\varrho_2)u_0\right]\\
	&\leq C_1\abs{\varrho_2 u_2-\varrho_0 u_0}\mbox{ since }\varrho_0\leq \varrho_2.
\end{align*}
Therefore, from \eqref{eqn-s2s0}, we get 
\begin{align*}
	\abs{\si}=s_2-s_1=s_2-s_0+s_0-s_1&=s_2-s_0+\abs{\si_2}\\
	&\leq 2(u_2-u_0)+\abs{\si_2}\\
	&\leq 2C_1\abs{\varrho_2u_2-\varrho_0u_0}+\abs{\si_2}.
\end{align*}This completes the proof of Lemma \ref{lemma-ibvp-2}.
\end{proof}
Recall that $\{\theta_n\}_{n\geq1}$ is randomly chosen sequence from $(-1,1)$. We set $\theta_0=0$. For initial boundary value problem, we define $I$-curve as follows: an  $I$-curve is a piece-wise linear, Lipschitz continuous function (contained in $[0,\f)\times[0,\f)$) such that each linear part coincides with one of the following. (see Figure \ref{fig-7})
\begin{enumerate}[(i)]
	\item the line joining $(n\De t,x_i+\theta_n\De x), ((n+1)\De t,x_i+\De x+\theta_{n+1}\De x)$ for $i\geq1$ and $n+i$ is even,
	\item the line joining $(k\De t,x_j+\theta_k\De x), ((k+1)\De t,x_j-\De x+\theta_{k+1}\De x)$ for $j\geq1$ and $k+j$ is even,
	\item the line joining $\left(\left(n+\frac{1}{2}\right)\De t,0\right)$ and $(n\De t,x_1+\theta_n\De x)$ when $n$ is odd,
	\item the line joining $\left(\left(n-\frac{1}{2}\right)\De t,0\right)$ and $(n\De t,x_1+\theta_n\De x)$ when $n$ is odd.
\end{enumerate}
	\begin{figure}[ht]
	\centering
	\begin{tikzpicture}
		
		\draw[thick][color=black] (-6.5 ,0) -- (4,0);
		\draw[thick][color=black] (-6.5 ,2) -- (4,2);
		\draw[thick][color=black] (-6.5 ,4) -- (4,4);
		\draw[thick][color=black] (-6.5 ,6) -- (4,6);
		
		\draw[thick][color=black] (-6.5 ,2) -- (-5.9,4);
		\draw[thick][color=black] (-6.5 ,0) -- (-6.5,6.5);
		\draw[thick][color=black] (-5 ,0) -- (-5,6.5);
		\draw[thick][color=black] (-3.5 ,0) -- (-3.5,6.5);
		\draw[thick][color=black] (-2 ,0) -- (-2,6.5);
		\draw[thick][color=black] (-0.5 ,0) -- (-0.5,6.5);
		\draw[thick][color=black] (1 ,0) -- (1,6.5);
		\draw[thick][color=black] (2.5,0) -- (2.5,6.5);
		
		\draw[thick][color=black] (-2,0) -- (-3.1,2);
		\draw[thick][color=black] (-2,0) -- (-1.2,2);
		\draw[thick][color=black] (-2,0) -- (-1.3,2);
		\draw[thick][color=black] (-2,0) -- (-1.4,2);
		\draw[thick][color=black] (-2,0) -- (-1.5,2);
		
		\draw[thick][color=black] (-5,0) -- (-4.4,2);
		\draw[thick][color=black] (-5,0) -- (-4.5,2);
		\draw[thick][color=black] (-5,0) -- (-4.6,2);
		\draw[thick][color=black] (-5,0) -- (-4.7,2);
		
		\draw[thick][color=black] (-5,4) -- (-4.2,6);
		\draw[thick][color=black] (-2,4) -- (-3,6);
		\draw[thick][color=black] (-2,4) -- (-3.1,6);
		\draw[thick][color=black] (-2,4) -- (-2.9,6);
		\draw[thick][color=black] (-2,4) -- (-1.5,6);
		\draw[thick][color=black] (-2,4) -- (-1.4,6);
		\draw[thick][color=black] (1,4) -- (0.5,6);
		\draw[thick][color=black] (1,4) -- (1.7,6);
		\draw[thick][color=black] (1,4) -- (1.5,6);
		\draw[thick][color=black] (1,4) -- (1.6,6);

		\draw[thick][color=black] (1,0) -- (0.2,2);
		\draw[thick][color=black] (1,0) -- (1.3,2);
		
		\draw[thick,dashed][color=blue] (-6.5,1)--(-5.4 ,2) -- (-3.5,0)--(-2.4,2)--(-0.5,0)--(0.6,2)--(2.5,0)--(3.5,1) ;
		\filldraw (-6.5,0) circle (1.5pt);
		\filldraw (-5,0) circle (1.5pt);
			\filldraw (-2,0) circle (1.5pt);
				\filldraw (1,0) circle (1.5pt);
		\filldraw (-5.4,2)  circle (1.5pt);
		\filldraw (-6.5,1)  circle (1.5pt);
		\filldraw (-3.5,0) circle (1.5pt);
		\filldraw (-2.4,2) circle (1.5pt);
		\filldraw (-0.5,0) circle (1.5pt);
		\filldraw (0.6,2) circle (1.5pt);
		\filldraw (2.5,0) circle (1.5pt);
		
		\draw[thick][color=black] (-3.5,2) -- (-4.2,4);
		\draw[thick][color=black] (-3.5,2) -- (-4.3,4);
		\draw[thick][color=black] (-3.5,2) -- (-4.4,4);
		\draw[thick][color=black] (-3.5,2) -- (-4.5,4);
		\draw[thick][color=black] (-3.5,2) -- (-2.9,4);
		\draw[thick][color=black] (-0.5,2) -- (0.5,4);
		\draw[thick][color=black] (-0.5,2) -- (-1.1,4);
		\draw[thick][color=black] (-0.5,2) -- (-1.2,4);
		\draw[thick][color=black] (-0.5,2) -- (-1.3,4);
		
			\filldraw (-6.5,2) circle (1.5pt);
				\filldraw (-6.5,4) circle (1.5pt);
		
		\draw[thick][color=black] (2.5,2) -- (2.9,4);
		\draw[thick][color=black] (2.5,2) -- (1.7,4);
		
		\draw[thick,dashed][color=blue] (-6.5,3)--(-5.4 ,2)--(-3.3,4)--(-2.4,2)--(-0.3,4)--(0.6,2)--(2.7,4)--(3.5,3) ;
		\draw[thick,dashed][color=blue] (-6.5,5)--(-4.8 ,6)--(-3.3,4)--(-1.8,6)--(-0.3,4)--(1.2,6)--(2.7,4)--(3.5,5) ;
		\filldraw (-6.5 ,3)  circle (1.5pt);
		\filldraw (-3.3,4) circle (1.5pt);
		\filldraw (-0.3,4) circle (1.5pt);
		\filldraw (2.7,4) circle (1.5pt);
		\filldraw (-6.5 ,5)  circle (1.5pt);
		\filldraw (-6.5,6)  circle (1.5pt);
		\filldraw (-4.8 ,6)  circle (1.5pt);
		\filldraw (-1.8,6)  circle (1.5pt);
		\filldraw (1.2,6)  circle (1.5pt);

		\draw[thick] (-6.5,0) node[anchor=north] {\tiny$(x_0,0)$};
		\draw[thick] (-5,0) node[anchor=north] {\tiny$(x_1,0)$};
		\draw[thick] (-3.5,0) node[anchor=north] {\tiny$(x_2,0)$};
		\draw[thick] (-2,0) node[anchor=north] {\tiny$(x_3,0)$};
		\draw[thick] (-0.5,0) node[anchor=north] {\tiny$(x_4,0)$};
		\draw[thick] (1,0) node[anchor=north] {\tiny$(x_5,0)$};
		\draw[thick] (2.5,0) node[anchor=north] {\tiny$(x_6,0)$};
		\draw[thick] (-6.5,2) node[anchor=east] {\tiny$(0,t_1)$};
		\draw[thick] (-6.5,4) node[anchor=east] {\tiny$(0,t_2)$};
		\draw[thick] (-6.5,1) node[anchor=east] {\tiny$\left(0,\frac{t_1}{2}\right)$};
		\draw[thick] (-6.5,3) node[anchor=east] {\tiny$\left(0,\frac{t_1+t_2}{2}\right)$};
		\draw[thick] (-6.5,5) node[anchor=east] {\tiny$\left(0,\frac{t_2+t_3}{2}\right)$};
		\draw[thick] (-6.5,6) node[anchor=east] {\tiny$\left(0,t_3\right)$};
	    \fill[white!40!white] (-6,2.1)rectangle (-5.4,2.5);
		\draw[thick] (-5.35,1.95) node[anchor=south east] {\tiny$(a^1_1,t_1)$};
		\draw[thick] (-2.35,1.9) node[anchor=south east] {\tiny$(a^1_3,t_1)$};
		\draw[thick] (0.65,1.9) node[anchor=south east] {\tiny$(a^1_5,t_1)$};
		
		\fill[white!40!white] (-2.8,4.05)rectangle (-3.5,4.3);
		\draw[thick] (-3.05,3.95) node[anchor=south] {\tiny$(a^2_2,t_2)$};
		
		\fill[white!40!white] (-0.15,4.05)rectangle (0.85,4.3);
		\draw[thick] (-0.05,4) node[anchor=south] {\tiny$(a^2_4,t_2)$};
		
		\fill[white!40!white] (2.7,4.05)rectangle (3.25,4.35);
		\draw[thick] (2.95,4) node[anchor=south] {\tiny$(a^2_6,t_2)$};
		
		\draw[thick] (-4.55,6) node[anchor=south] {\tiny$(a^3_1,t_3)$};
		\draw[thick] (-1.55,6) node[anchor=south] {\tiny$(a^3_3,t_3)$};
		\draw[thick] (1.45,6) node[anchor=south] {\tiny$(a^3_5,t_3)$};
		
	\end{tikzpicture}
	\caption{This illustrates the $I$-curves (in dotted lines) formed by the lines joining $(a^j_{i-1},t_j),(a^{j+1}_{i+1},t_{j+1})$ and $(a^j_{i-1},t_j),(a^{j-1}_{i+1},t_{j-1})$ where $a^i_j=x_j+\theta_i\De x$ and $x_j=j\De x$.}
	\label{fig-7}
\end{figure}
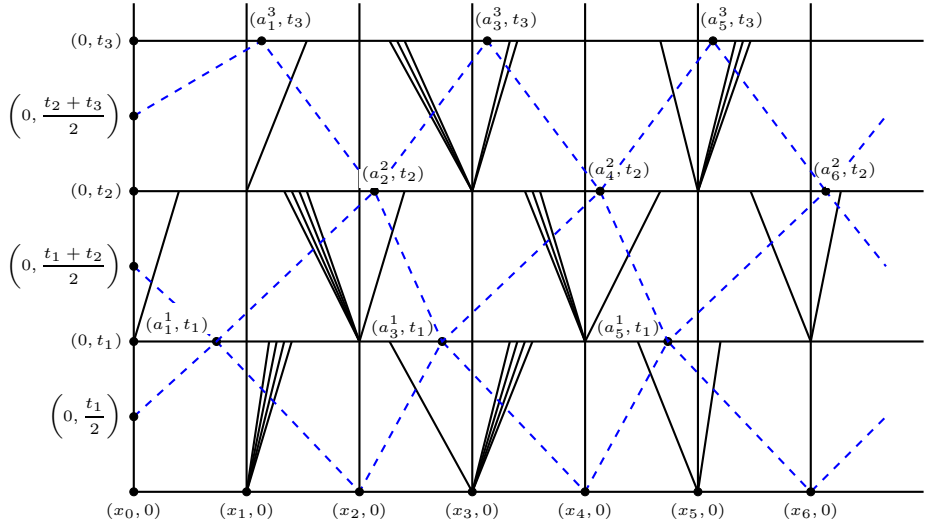

Similar to section \ref{sec:Glimm}, we define $O^n$-curve is an $I$-curve contained in $\{(x,t); x\geq0,n\De t\leq t\leq (n+1)\De t\}$, $n\geq0$. We consider $V$ as in \eqref{def:V} and adapting \cite{NS-bdy}, we modify the quadratic part $Q$ as follows.
\begin{equation*}
	Q_{b}=\sum\left\{\abs{\beta}\abs{\gamma}: \beta,\gamma\;\mbox{ cross $J$ and approach}\right\}+\sum\left\{\abs{\B}^2:\,\B\mbox{ is 1-shock}\right\}.
\end{equation*}
Similar to \eqref{def:F}, we define $F_b(J)=V(J)+KQ_b(J)$. We remark that for well-chosen $K$ the functional $F_b$ is non-increasing after a homogeneous interaction, that is, Lemma \ref{lemSmo} remains true at any interaction which is away from  boundary. Consider the notion $\tilde{J}$ as in Lemma \ref{lemSmo}, we can show that for sufficiently small $\e>0$ we have $F_b(\tilde{J})\leq F_b(J)$ when $\tilde{J}\setminus J$ is away from boundary. Note that here we have changed the quadratic part by adding the square of 1-shock strengths. This does not add any difficulty in estimating $F_b(\tilde{J})\leq F_b(J)$ when $\tilde{J}\setminus J$ is away from boundary and we omit the details here (see \cite{NS-bdy}). Therefore, to remove the interaction at boundary let us consider the following polygonal lines when $n$ is odd (similar argument follows for $n$ even) 
\begin{align*}
	J^{n}_{0}:&\left(\frac{2n+1}{2}\De t,0\right)\rr(n\De t,\De x+\theta_n\De x)\rr ((n+1)\De t,2\De x+\theta_{n+1}\De x),\\
	J^{n+1}_0:&\left(\frac{2n+3}{2}\De t,0\right)\rr((n+2)\De t,\De x+\theta_{n+2}\De x)\rr ((n+1)\De t,2\De x+\theta_{n+1}\De x).
\end{align*}
Note that from $\mathcal{O}^n$ to $J_0^n\cup\left(\mathcal{O}^{n+1}\setminus J_0^{n+1}\right)=:\mathcal{O}^{n+1}_1$ it does not count the effect of boundary. Hence, by using similar argument as in section \ref{sec:Glimm} we obtain
\begin{equation*}
	F_b(O^{n+1}_1)\leq F_b(O^n )(1+A_1\De t)+B_1\De t,
\end{equation*}
where the constants $A_1,B_1$ are as in \eqref{def:A1}--\eqref{def:B1}. Our BV estimation will be completed if we are able to make a bound of $F_b(O^{n+1})$ with respect to $F_b(O^{n+1}_1)$ or in other words, we want to compare strength of $\mathcal{R}_b[m_b^{n+1},U_{n+1,2}]$ with the waves intersecting $J_0^n$. We can not do this directly. Instead, we rely on the following three steps: (i) compare $\mathcal{R}[(\varrho_*,u_*),U_{n+1,2}]$ with the waves intersecting $J_0^n$ where $(\varrho_*,u_*)=\mathcal{R}_b[m_b^n,U_{n,1}](\De t,0)$. This will be same as an interaction away from boundary. (ii) Compare the strength of 2-wave arising in $\mathcal{R}_b[m_b^n,U_{n+1,2}]$ with $\mathcal{R}[(\varrho_*,u_*),U_{n+1,2}]$ by using Lemma \ref{lemma-ibvp-1} and then finally (iii) we compare strength of 2-wave arising in $\mathcal{R}_b[m_b^{n+1},U_{n+1,2}]$ with the strength of 2-wave arising in $\mathcal{R}_b[m_b^n,U_{n+1,2}]$ by using Lemma \ref{lemma-ibvp-2} . Next we justify these steps in more details.\\

\noi\underline{Step-(i):} We first consider a Riemann problem with $U_L=U_*$ where $U_*=(\varrho_*,u_*)=\mathcal{R}_b[m_b^n,U_{n,1}](\De t,0)$ and $U_{n+1,2}=(\varrho_{n+1,2},u_{n+1,2})$. Suppose the Riemann problem $\mathcal{R}[U_*,U_{n+1,2}]$ consists of 1-wave (say $\al_1$) of strength $\si_1^+$ and 2-wave (say $\al_2$) of strength $\si_2^+$. Let ${F}_*^{n+1}$ be defined as 
\begin{align*}
	F^{n+1}_*&:=(\si_1^+)_-+(\si_2^+)_-+K(\si_1^+)_-^2+V(\mathcal{O}^{n+1}\setminus J_0^{n+1})\\
	&+K\sum\limits\left\{\abs{\B}(\si_1^+)_-,\,\B\mbox{ is a shock approaching }\al_1\mbox{ and $\B$ is on $\mathcal{O}^{n+1}\setminus J_0^{n+1}$}\right\}\\
	&+K\sum\limits\left\{\abs{\B}(\si_2^+)_-,\,\B\mbox{ is a shock approaching }\al_2\mbox{ and $\B$ is on $\mathcal{O}^{n+1}\setminus J_0^{n+1}$}\right\}\\
	&+K\sum\left\{\abs{\beta}\abs{\gamma}: \beta,\gamma\;\mbox{ cross $J$ and approach and $\B,\ga$ are on $\mathcal{O}^{n+1}\setminus J_0^{n+1}$}\right\}\\
	&+K\sum\left\{\abs{\B}^2:\,\B\mbox{ is 1-shock and lies on $\mathcal{O}^{n+1}\setminus J_0^{n+1}$}\right\}.
\end{align*}
Again by similar argument as in section \ref{sec:Glimm} (see also \cite{NS-bdy}), we get  
\begin{equation*}
	F^{n+1}_*\leq F_b(\mathcal{O}^{n+1}_1)\leq F_b(O^n)(1+A_1\De t)+B_1\De t.
\end{equation*}

\noi\underline{Step-(ii):} Now we consider the boundary Riemann problem $\mathcal{R}_b[m_b^n,U_{n+1,2}]$ with boundary condition $m_b^n$. Let $(\varrho^\#,u^\#)$ be defined as $(\varrho^\#,u^\#)=\mathcal{R}_b[m_b^n,U_{n+1,2}](\De t,0)$ which implies $\varrho^\# u^\#=m_b^n$. Let $\si^\#$ be the strength of 2-wave (say $\al^\#$) in $\mathcal{R}_b[m_b^n,U_{n+1,2}]$. We define
\begin{align*}
	F^{n+1}_\#&:=(\si^\#)_-+V(\mathcal{O}^{n+1}\setminus J_0^{n+1})\\
	&+K\sum\limits\left\{\abs{\B}(\si^\#)_-,\,\B\mbox{ is a shock approaching }\al^\#\mbox{ and $\B$ is on $\mathcal{O}^{n+1}\setminus J_0^{n+1}$}\right\}\\
	&+K\sum\left\{\abs{\beta}\abs{\gamma}: \beta,\gamma\;\mbox{ cross $J$ and approach and $\B,\ga$ are on $\mathcal{O}^{n+1}\setminus J_0^{n+1}$}\right\}\\
	&+K\sum\left\{\abs{\B}^2:\,\B\mbox{ is 1-shock and lies on $\mathcal{O}^{n+1}\setminus J_0^{n+1}$}\right\}.
\end{align*}By applying Lemma \ref{lemma-ibvp-1}, we see that
\begin{equation*}
	F^{n+1}_\#\leq F^{n+1}_*\leq F_b(O^n)(1+A_1\De t)+B_1\De t.
\end{equation*}

\noi\underline{Step-(iii):} Finally, we consider the boundary Riemann problem $\mathcal{R}_b[m_b^{n+1},U_{n+1,2}]$  and compare it with $\mathcal{R}_b[m_b^n,U_{n+1,2}]$. Let $\si^{\ddagger}$ be the strength of 2-wave (say $\al^\ddagger$) arising in $\mathcal{R}[m_b^{n+1},U_{n+1,2}]$, then by Lemma \ref{lemma-ibvp-2} we have
\begin{align*}
	F_b(\mathcal{O}^{n+1})&=(\si^\ddagger)_-+V(\mathcal{O}^{n+1}\setminus J_0^{n+1})\\
	&+K\sum\limits\left\{\abs{\B}(\si^\ddagger)_-,\,\B\mbox{ is a shock approaching }\al^\ddagger\mbox{ and $\B$ is on $\mathcal{O}^{n+1}\setminus J_0^{n+1}$}\right\}\\
	&+K\sum\left\{\abs{\beta}\abs{\gamma}: \beta,\gamma\;\mbox{ cross $J$ and approach and $\B,\ga$ are on $\mathcal{O}^{n+1}\setminus J_0^{n+1}$}\right\}\\
	&+K\sum\left\{\abs{\B}^2:\,\B\mbox{ is 1-shock and lies on $\mathcal{O}^{n+1}\setminus J_0^{n+1}$}\right\}\\
	&\leq F^{n+1}_\#\leq F_b(O^n)(1+A_1\De t)+B_1\De t. 
\end{align*}
Rest of the proof follows in exactly same way as in Theorem \ref{theorem-1}.

\bigskip
\noindent\textbf{Acknowledgement.} Authors thank the IFCAM project ``Conservation laws: $BV^s$, interface and control". SSG would like to express thanks to the Department of Atomic Energy, Government of India, under project no. 12-R\&D-TFR-5.01-0520 for support. SSG acknowledge the support of Inspire faculty-research grant DST/INSPIRE/04/2016/000237. AJ acknowledges the support of Fondation Sciences Math\'ematiques de Paris-FSMP. BH has been partially funded by the ANR project INFAMIE ANR-15-CE40-0011.

\end{document}